\documentclass[]{siamart0516}

\usepackage{hyperref}

\usepackage{lipsum}
\usepackage{amsfonts}
\usepackage{graphicx}
\usepackage{epstopdf}
\usepackage{algorithm} 
\usepackage{algpseudocode}
\usepackage{tikz}
\usepackage{pgfplots}

\numberwithin{equation}{section}
\numberwithin{theorem}{section}
\numberwithin{figure}{section}
\numberwithin{algorithm}{section}

\usetikzlibrary{external}
\newcommand{\includetikz}[1]{%
	\tikzsetnextfilename{Figures/#1}%
	\input{Figures/#1.tikz}%
}

\ifpdf
  \DeclareGraphicsExtensions{.eps,.pdf,.png,.jpg}
\else
  \DeclareGraphicsExtensions{.eps}
\fi

\newcommand{\TheTitle}{ Balanced Truncation Model Order Reduction for Quadratic-Bilinear Control Systems}
\newcommand{\TheShortTitle}{ Balanced Truncation MOR for Quadratic-Bilinear  Systems}
\newcommand{\TheAuthors}{Peter Benner, and Pawan Goyal}

\headers{\TheShortTitle}{\TheAuthors}

\title{{\TheTitle}
}

\author{
  Peter Benner\thanks{Max Planck Institute for Dynamics of Complex Technical Systems, Sandtorstra\ss e 1, 39106 Magdeburg, Germany
    (\email{benner@mpi-magdeburg.mpg.de}).}
  \and
  Pawan Goyal\thanks{Corresponding author. Max Planck Institute for Dynamics of Complex Technical Systems, Sandtorstra\ss e 1, 39106 Magdeburg, Germany (\email{goyalp@mpi-magdeburg.mpg.de}).}
}

\usepackage{amsopn}
\usepackage[software,hardware]{mymacros}

\newcommand{\bbm}{\begin{bmatrix}}
	\newcommand{\ebm}{\end{bmatrix}}
\newcommand{\bpm}{\begin{pmatrix}}
	\newcommand{\epm}{\end{pmatrix}}

\newenvironment{customlegend}[1][]{%
	\begingroup
	\csname pgfplots@init@cleared@structures\endcsname
	\pgfplotsset{#1}%
}{%
\csname pgfplots@createlegend\endcsname
\endgroup
}%

\def\addlegendimage{\csname pgfplots@addlegendimage\endcsname}

\newlength\figureheight
\newlength\figurewidth
\newlength\fheight
\newlength\fwidth

\usepackage{adjustbox}
\usepackage{pdflscape}

\usepackage[ntheorem, fleqn]{empheq}
\usepackage[american]{babel}
\usepackage[title,titletoc,toc]{appendix}

\usepackage{caption}
\usepackage{subcaption}
\usepackage{nicefrac}
\newtheorem{example}[theorem]{Example}
\newtheorem{remark}[theorem]{Remark}
\usepackage{enumitem}
\renewcommand\hat[1]{\widehat#1}
\renewcommand\tilde[1]{\widetilde#1}

\ifpdf
\hypersetup{
  pdftitle={\TheTitle},
  pdfauthor={\TheAuthors}
}
\fi

\begin{document}

\maketitle

\begin{abstract}
  We discuss balanced truncation  model order reduction for large-scale quadratic-bilinear (QB) systems. Balanced truncation for linear systems mainly involves the computation of the Gramians of the system, namely  reachability and observability Gramians. These Gramians are extended to a general nonlinear setting in Scherpen (1993), where it is shown that Gramians for nonlinear systems are the solutions of state-dependent nonlinear Hamilton-Jacobi equations. Therefore, they are not only  difficult to compute for large-scale systems but also hard to utilize in the model reduction framework. In this paper, we propose algebraic Gramians for  QB systems based on the underlying Volterra series representation of QB systems and their Hilbert adjoint  systems. We then show their relations with a certain type of generalized quadratic Lyapunov equation. Furthermore, we present how these algebraic Gramians and energy functionals relate to each other. Moreover, we  characterize the reachability and observability of QB systems based on the proposed algebraic Gramians.  This allows us to find those states that are hard to control and hard to observe via an appropriate transformation based on the Gramians. Truncating such states yields reduced-order systems.  Additionally, we present a truncated version of the Gramians for QB systems and discuss their advantages in the model reduction framework.  We also investigate the Lyapunov stability of the reduced-order systems. We finally illustrate the efficiency of  the proposed balancing-type model reduction for QB systems by means of various semi-discretized nonlinear partial differential equations and show its competitiveness with the existing moment-matching methods for QB systems.
\end{abstract}



\section{Introduction}   
Numerical simulations are considered to be a primary tool in studying dynamical systems, e.g., in prediction and control studies. High-fidelity modeling is an  essential step to gain deep insight into the behavior of complex  dynamical systems.  Even though computational resources have been developing extensively over the last few decades,  fast numerical simulations of such high-fidelity systems, whose number of state variables can easily be of order $\cO(10^5){-}\cO(10^6)$,  are still a huge computational burden. This makes the usage of these large-scale systems very difficult and inefficient, for instance,  in optimization and control design. One approach to mitigate this problem is \emph{model order reduction} (MOR). MOR seeks to substitute these large-scale dynamical systems by  low-dimensional (reduced-order) systems such that the input-output behaviors of  both original and reduced-order systems are close enough, and the reduced-order systems preserve some important properties, for instance, stability and passivity of the original system.

MOR techniques and strategies for linear systems are well-established and  are widely applied in various application areas, see, e.g.,~\cite{morAnt05,morBenMS05,morSchVR08}. In many applications, where the dynamics are governed by nonlinear PDEs, such as Navier-Stokes equations, Burgers' equations, a linear system can also be obtained via  linearization of the system around a suitable expansion point, e.g., the steady-state solution.   Notwithstanding the linearized system captures the dynamics very well locally. However, as it moves away from the expansion point, the linearized system might not be able to capture the system dynamics accurately. Therefore, there is a general need to take nonlinear terms into consideration, thus resulting in  a more accurate system. Consider a nonlinear system of the form
\begin{equation}\label{eq:NonlinearSys}
\begin{aligned}
 \dot{x}(t) &=  f(x(t)) + g(x(t),u(t)), \\
 y(t) &= h(x(t),u(t)), \qquad x(0) = x_0,
\end{aligned}
 \end{equation}
where $f:\Rn \rightarrow \Rn$, $g:\Rn\times\Rm \rightarrow \Rn$ and $h:\Rn\times \Rm \rightarrow \Rp$ are nonlinear state-input evolution functions, and $x(t)\in\Rn, u(t)\in\Rm$ and $y(t)\in \Rp$ denote the state, input and output vectors of the system at time $t$, respectively.  Also, we consider a fixed initial condition $x_0$ of the system. However,   without loss of generality,  we assume  the zero initial condition, i.e., $x(0) =0$. In  case $x(0) \neq 0$, one can transform the system by defining new appropriate state variables as $\tx(t) = x(t) - x_0$  to ensure  a zero initial condition of the transformed system, e.g., see \cite{morBauBF14}. The main goal of MOR is to construct a low-dimensional system, having  a similar form as the system~\eqref{eq:NonlinearSys}
\begin{equation}\label{eq:NonlinearSys_Red}
\begin{aligned}
 \dot{\hx}(t) &=  \hf(\hx(t)) + \hg(\hx(t),u(t)), \\
 \hy(t) &= \hh(\hx(t),u(t)), \qquad \hx(0) = 0,
\end{aligned}
 \end{equation}
in which $\hf:\R^{\hn} \rightarrow \R^{\hn}$, $\hg:\R^{\hn}\times\Rm \rightarrow \R^{\hn}$ and $\hh:\R^{\hn}\times \Rm \rightarrow \Rp$ with $\hn \ll n$ that fulfills our desired requirements.

MOR techniques  for general nonlinear systems, namely trajectory-based MOR techniques, have been widely applied in the literature to determine  reduced-order systems for nonlinear systems; see, e.g.,~\cite{morAstWWetal08,morChaS10,morHinV05}. The proper orthogonal decomposition (POD) method is a very powerful trajectory-based MOR technique, which depends on a Galerkin projection $\cP = VV^T$, where $V$ is a projection matrix such that $x(t) \approx V\hx(t)$. The nonlinear functions $\hf(\hx)$ can be given as $\hf(\hx(t)) = V^Tf(V\hx(t))$, and  similar expressions  can also be derived for $\hg(\hx(t),u(t))$ and $\hh(\hx(t),u(t))$. This method preserves the structure of the original system in the reduced-order system, but still, the reduced-order system requires the computation of the nonlinear functions on the full grid. This may obstruct the success  of MOR;  however, there are many new advanced methodologies such as the empirical interpolation method (EIM), the discrete empirical  interpolation method (DEIM), the best point interpolation method (BPIM),  to perform the computation of the nonlinear functions cheaply and quite accurately. For details, we refer to~\cite{morBarMNetal04,morChaS10,drmac2016new,morGreMNetal07}.

Another popular trajectory-based MOR technique is based on trajectory piecewise linearization (TPWL)~\cite{morRew03}, where nonlinear functions are replaced by a weighted combination of linear systems. These linear systems  can then  be reduced by applying  well-established MOR techniques for linear systems such as balanced truncation or the interpolation-based iterative method (IRKA); see, e.g.,~\cite{morAnt05,morGugAB08}. In recent years, reduced basis methods have been successfully applied to nonlinear systems to obtain  reduced-order systems~\cite{morBarMNetal04,morGreMNetal07}. In spite of all these, the trajectory-based MOR techniques have the drawback  of being input dependent. This makes the obtained reduced-order systems inadequate to  control applications, where the input function may vary significantly from any used training input. 

In this article, we consider a certain class of  nonlinear control systems, namely quadratic-bilinear (QB) control systems. The advantage of considering this special class of  nonlinear systems is that systems, containing smooth mono-variate nonlinearities such as exponentials, polynomials, trigonometric functions, can also be rewritten in the QB form by introducing some new variables in the state vector~\cite{morGu09}.  Note that this transformation is exact, i.e., it requires no approximation and does not introduce any error, but this transformation may not be unique. 

Related to MOR for QB systems, the idea of one-sided  moment-matching has been extended from linear or bilinear systems to QB systems; see, e.g.,~\cite{bai2002krylov,feng2004direct,morGu09,li2005compact,morPhi03}, where a reduced system is determined by capturing the input-output behavior of the original system,   given by generalized transfer functions. More recently, this has been extended to two-sided moment-matching in~\cite{morBenB15}, ensuring more  moments to be matched, for a given order of the reduced system. Despite these methods have evolved as an effective MOR technique for nonlinear systems in recent times, shortcomings of these methods are: how to choose an appropriate order of the reduced system and how to select good interpolation points. Moreover, the applicability of the two-sided moment-matching method~\cite{morBenB15} is limited to single-input single-output QB systems, and also the stability of the obtained reduced-order system is a major issue in this method.

In this article, our focus rather lies on balancing-type MOR techniques for QB systems. This technique mainly depends on controllability and observability energy functionals, or in other words, Gramians of the system.  This method is presented for linear systems, e.g., in~\cite{morAnt05,morMoo81}, and later on, a theory of balancing for general nonlinear systems is developed in a sequence of papers~\cite{morfuji10,morgray2001,morSch93,morsche1996h,morsche94nor}.  In the general nonlinear case, the balancing requires  the solutions of the state-dependent nonlinear Hamilton-Jacobi equation which are, firstly, very expensive to solve for large-scale dynamical systems; secondly, it is not  straightforward  to use them in the MOR context. Along with these, it may happen that the reduced-order systems, obtained from nonlinear balancing, do not preserve the structure of the nonlinearities in the system. However, for some  weakly nonlinear systems,  the so-called bilinear systems, reachability and observability Gramians have been  studied in~\cite{moral1994,morBenD11,bennertruncated,morcondon2005,enefungray98}, which are solutions to generalized algebraic Lyapunov equations. Moreover, these Gramians, when used to define appropriate  quadratic forms, approximate  energy functionals of bilinear systems (in the neighborhood of the origin), see~\cite{morBenD11, bennertruncated}

Moving in the direction of balancing-type  MOR for QB systems, our first goal is to come up with reachability and observability Gramians for these systems, which are state-independent  matrices and suitable for the MOR purpose. In addition to this, we need to show  how the Gramians  relate to the energy functionals of the QB systems and provide interpretations of reachability and observability of the  system with respect to these Gramians. To this end, in the subsequent section, we review  background material  associated with energy functionals and a duality of the nonlinear systems, which  serves as the basis for the rest of the paper. In \Cref{sec:gramians}, we propose the  reachability Gramian  and its truncated version for QB systems based on the underlying Volterra series of the system. Additionally, we determine the observability Gramian and its truncated version based on the dual system associate to the QB system. Furthermore, we establish  relations between the solutions of a certain type of quadratic Lyapunov equations and these Gramians.  In \Cref{sec:energyfunctionals},  we develop the connection  between the proposed  Gramians  and the energy functionals of the QB systems and  reveal their relations to reachability and observability of the system. Consequently, we utilize these Gramians for balancing of QB systems, allowing us to determine those states that are hard to control as well as hard to observe. Truncation of such states leads to reduced systems.  In \Cref{sec:computational}, we discuss the related computational issues and advantages of the truncated version of Gramians in the MOR framework. We further discuss the stability of these reduced systems.  In \Cref{sec:numerics}, we test the efficiency of the proposed balanced truncation MOR technique for various semi-discretized nonlinear PDEs and compare it with the existing moment-matching techniques for the QB systems.

\section{Preliminaries} \label{sec:background}We begin with recapitulation of energy functionals for nonlinear systems.
\subsection{Energy functionals for nonlinear systems}
In this subsection, we give a brief overview of energy functionals, namely controllability and observability energy functionals for nonlinear systems. For this, let us consider the following smooth, for example, $C^\infty$, nonlinear asymptotically stable input-affine nonlinear system of the form
\begin{equation}\label{eq:Gen_NonlinearSys}
 \begin{aligned}
  \dot{x}(t) &=  f(x) + g(x)u(t),\\
  y(t) &= h(x),\qquad x(0) = 0,
 \end{aligned}
\end{equation}
where $x(t) \in \Rn$, $u(t) \in \Rm$ and $y(t)\in\Rp$ are the state, input and output vectors of the system, respectively, and also $f(x) : \Rn \rightarrow \Rn$, $g(x) :\Rn \rightarrow \R^{n\times m}$ and $h(x) : \Rn\rightarrow \Rp$ are smooth nonlinear functions. Without loss of generality, we assume  that $0$ is an equilibrium of  the system~\eqref{eq:Gen_NonlinearSys}.  The controllability and observability energy functionals for the general nonlinear systems  first have been studied in the literature in~\cite{morSch93}. In the following, we state the definitions of controllability and observability energy functionals for the system~\eqref{eq:Gen_NonlinearSys}.
\begin{definition}\cite{morSch93}
 The controllability energy functional is defined as the minimum amount of energy required to steer the system from $x(-\infty) = 0$ to $x(0) = x_0$:
 \begin{equation*}
  L_c(x_0) = \underset{\begin{subarray}{c}
 u\in L^m_2(-\infty,0] \\[5pt]
 x(-\infty)=0,~x(0) = x_0
  \end{subarray}  }{\min} \dfrac{1}{2}\int_{-\infty}^0 \|u(t)\|^2dt.
 \end{equation*}
\end{definition}
\begin{definition}\cite{morSch93}\label{def:obser1}
The observability energy functional can be defined as the  energy generated by the nonzero initial condition $x(0) = x_0$ with zero control input:
\begin{equation*}
 L_o(x_0) = \dfrac{1}{2}\int_0^\infty \|y(t)\|^2dt.
 \end{equation*}
\end{definition}
We assume that the system~\eqref{eq:Gen_NonlinearSys} is controllable and observable. This implies that the system~\eqref{eq:Gen_NonlinearSys} can always be steered from $x(0) =0$ to $x_0$ by using appropriate inputs.
To define the observability energy functional (\Cref{def:obser1}), it is assumed that the nonlinear system~\eqref{eq:Gen_NonlinearSys} is  a zero-state observable. It means that if $u(t) = 0$ and $y(t) =0$ for $t\geq 0$, then $x(t) = 0$ $\forall t\geq 0$. However, as discussed in~\cite{morgray1996}, for a nonlinear system such a condition can be very strong. As a result, therein, it is shown how this condition can be  relaxed in the context of general input balancing, and a new definition for the observability functionals was proposed as follows:

\begin{definition}\cite{morgray1996}\label{def:obser2}
	The observability energy functional can be defined as the  energy generated by the nonzero initial condition $x(0) = x_0$  and by applying an $L^m_2$-bounded input:
	 \begin{equation*}
	  L_o(x_0) = \underset{\begin{subarray}{c}
	  u\in L^m_2[0,\infty), \|u\|_{L_2} \leq \alpha \\[5pt]
	  x(0) = x_0,x(\infty)=0
	   \end{subarray}  }{\max} \dfrac{1}{2}\int_0^\infty \|y(t)\|^2dt.
	  \end{equation*}
\end{definition}

In an abstract way, the main idea of introducing \Cref{def:obser2} to find the state component that contributes less from a state-to-output point of view for all possible $L_2$-bounded inputs. 
The connections between these energy functionals and the solutions of the partial differential equations  are established in~\cite{morgray1996,morSch93}, which are outlined in the following theorem.
\begin{theorem}\cite{morgray1996,morSch93}\label{thm:energy_function}
Consider the nonlinear system~\eqref{eq:Gen_NonlinearSys}, having $x = 0$ as an asymptotically stable equilibrium of the system in a neighborhood $W_o$ of $0$. Then, for all $x\in W_o$, the observability energy functional $L_o(x)$ can be determined  by the following partial differential equation:
\begin{equation}\label{eq:Obser_Diff}
 \dfrac{\partial L_o}{\partial x}  f(x) + \dfrac{1}{2}h^T(x)h(x) - \dfrac{1}{2}\mu^{-1}\dfrac{\partial L_o}{\partial x}g(x)g(x)^T\dfrac{\partial^T L_o}{\partial x} = 0,\quad L_o(0) = -\dfrac{1}{2}\mu,
\end{equation}
assuming that there exists a smooth solution $\bar{L}_o$ on $W$, and $0$ is an asymptotically stable equilibrium of $\bar{f}:= (f-\mu^{-1}gg^T\tfrac{\partial^T \bar{L}_o} {\partial x})$ on $W$  with a negative real number $\mu:= -\|g^T(\phi)\tfrac{\partial^T \bar{L}_o} {\partial x} (\phi) \|_{L_2}$, and $\dot{\phi} = \bar{f}(\phi)$ with $\phi(0) = x$.
 Moreover, for all $x\in W_c$, the controllability energy functional $L_c(x)$ is a unique smooth solution of the following Hamilton-Jacobi equation:
\begin{equation}\label{eq:Cont_Diff}
 \dfrac{\partial L_c}{\partial x} f(x) +  f(x)\dfrac{\partial L_c}{\partial x} + \dfrac{\partial L_c}{\partial x}g(x)g(x)^T\dfrac{\partial^T L_c}{\partial x} = 0,\quad L_c(0) = 0
\end{equation}
under the assumption that~\eqref{eq:Cont_Diff} has a smooth solution $\bar{L}_c$ on $W_c$, and $0$ is an asymptotically stable equilibrium of $-\left(f(x) + g(x)g(x)^T\tfrac{\partial\bar{L}_c(x)}{\partial x}^T\right)$ on $W_c$.
\end{theorem}

Note that in \Cref{def:obser2}, the zero-state observable condition is relaxed by considering $L_2$-bounded inputs.  However, an alternative way to relax the zero-state observable condition by considering not only $L_2$-bounded inputs but also  $L_\infty$ bounded inputs. We thus propose a new definition of the observability energy functional as follows:

\begin{definition}\label{def:obser3}
	The observability energy functional can be defined as the  energy generated by the nonzero initial condition $x(0) = x_0$ and by applying an $L_2$-bounded and $L_\infty$-bounded input:
	\begin{equation*}
	L_o(x_0) = \underset{\begin{subarray}{c}
		u\in \cB_{(\alpha,\beta)} \\[5pt]
		x(0) = x_0,x(\infty)=0
		\end{subarray}  }{\max} \dfrac{1}{2}\int_0^\infty \|y(t)\|^2dt,
	\end{equation*}
\end{definition}
where $\cB_{(\alpha,\beta)} \overset{\mathrm{def}}{=} \{u \in L_2^m[0,\infty), \|u\|_{L_2}\leq \alpha, \|u\|_{L_\infty}\leq \beta \}$. In this paper, we use the above definition to characterize the observability energy functional for QB systems. 
\subsection{Hilbert adjoint operator for nonlinear systems}
 The importance of the adjoint operator (dual system) can be seen, particularly, in the computation of the observability energy functional or Gramian.
 For  general nonlinear systems, a duality between controllability and observability energy functionals is shown in~\cite{Adjfujimoto2002} with the help of  state-space realizations for nonlinear adjoint operators. In what follows, we briefly outline the state-space realizations for nonlinear adjoint operators of nonlinear systems. For this, we consider a nonlinear system of the form
 \begin{equation}\label{eq:Gen_Nonlinear}
\Sigma := \begin{cases}
\begin{aligned}
 \dot{x}(t) &= \cA(x,u,t) x(t) + \cB(x,u,t)u(t),\\
 y(t) &= \cC(x,u,t)x(t) + \cD(x,u,t)u(t),\qquad x(0) = 0
 \end{aligned}
\end{cases}
\end{equation}
in which $x(t) \in \Rn$, $u(t) \in \Rm$ and $y(t) \in \Rp$ are the state, input and output  vectors of the system, respectively, and $\cA(x,u,t)$, $\cB(x,u,t) $, $\cC(x,u,t)$ and $\cD(x,u,t) $ are appropriate size matrices.  Also, we assume that the origin is a stable equilibrium of the system. The Hilbert adjoint operators for the general nonlinear systems have been investigated in~\cite{Adjfujimoto2002}. Therein,  a connection between the state-space realization of the adjoint operators and  port-control Hamiltonian systems is also discussed, leading to the state-space characterization of the nonlinear Hilbert adjoint operators of $\Sigma:L_2^m(\Omega) \rightarrow L_2^p(\Omega)$. In the following lemma, we summarize the state-space realization of the Hilbert adjoint operator of the nonlinear system.
\begin{lemma}\cite{Adjfujimoto2002}\label{lemma:adjointsys}
 Consider the system~\eqref{eq:Gen_Nonlinear} with the initial condition $x(0) =0$, and assume that the input-output mapping $u\rightarrow y$ is denoted by the operator $\Sigma : L_2^m(\Omega) \rightarrow L_2^p(\Omega)$. Then, the state-space realization of the nonlinear Hilbert adjoint operator $\Sigma^*:L_2^{m+p}(\Omega) \rightarrow L_2^m(\Omega)$ is given by
 \begin{equation}\label{eq:Gen_Nonlinear_Adj}
 \Sigma^*(u_d,u) := \begin{cases}
\begin{aligned}
   \dot{x}(t) &= \cA(x,u,t)x(t) + \cB(x,u,t)u(t), & \quad  x(0) &= 0,\\
   \dot{x_d}(t) &= -\cA(x,u,t)x_d(t) - \cC^T(x,u,t)u_d(t), & \quad  x_d(\infty) &= 0,\\
   y_d(t) &= \cB^T(x,u,t)x_d(t) + \cD^T(x,u,t)u_d(t),
  \end{aligned}
\end{cases}
\end{equation}
where $x_d \in \Rn$, $u_d \in\Rp$ and $y_d\in\Rm$ can be interpreted as the dual state, dual input and  dual output vectors of the system, respectively.
\end{lemma}
We will see in the subsequent section the importance of the dual system in determining the observability energy functional or observability Gramian for a  QB system because a duality of the energy functionality holds.

\section{Gramians for QB Systems} \label{sec:gramians}This section is devoted to determine  algebraic Gramians for QB systems, which are also related to the energy functionals of the quadratic-bilinear systems as welI. Let us consider  QB systems of the form
\begin{subequations}\label{eq:Quad_bilin_Sys}
\begin{align}
 \dot{x}(t) &= Ax(t) + H~x(t)\otimes x(t) + \sum_{k = 1}^mN_kx(t)u_k(t) +  Bu(t),\label{eq:QB_differential}\\
 y(t) &= Cx(t),\quad x(0) = 0,\label{eq:QB_output}
 \end{align}
\end{subequations}
where $A,N_k\in \Rnn, H\in\R^{n\times n^2}, B \in \Rnm$ and $C\in \Rpn$. Furthermore, $x(t) \in \Rn$, $u(t)\in \Rm$ and $y(t) \in \Rp$ denote the state, input and output vectors of the system, respectively. Since the system~\eqref{eq:Quad_bilin_Sys} has a quadratic nonlinearity in the state vector $x(t)$ and also includes  bilinear terms $N_kx(t)u_k(t)$, which are products of the state vector and inputs, the system is called a quadratic-bilinear (QB) system. 
We begin by deriving  the reachability Gramian of  the QB system  and its connection with a certain type of  quadratic Lyapunov equation.


\subsection{Reachability Gramian for QB systems}
In order to derive the reachability Gramian, we first formulate the Volterra series for the QB system~\eqref{eq:Quad_bilin_Sys}. Before we proceed further, for ease we define the following short-hand notation:
\begin{equation*}
u^{(k)}_{\sigma_1,\ldots ,\sigma_l}(t) := u_k(t-\sigma_1 \cdots -\sigma_l)\quad\mbox{and}\quad x_{\sigma_1,\ldots ,\sigma_l}(t) := x(t-\sigma_1 \cdots -\sigma_l).
\end{equation*}
We integrate  both sides of the differential equation~\eqref{eq:QB_differential} in the state variables with respect to time  to obtain
\begin{multline}\label{eq:first_int}
x(t) = \int\nolimits_0^te^{A\sigma_1}Bu_{\sigma_1}(t)d\sigma_1 +  \sum_{k=1}^m \int\nolimits_0^te^{A\sigma_1} N_kx_{\sigma_1}(t)u^{(k)}_{\sigma_1}(t)d\sigma_1 \\ +\int\nolimits_0^te^{A\sigma_1}H\left(x_{\sigma_1}(t)\otimes x_{\sigma_1}(t)\right)d\sigma_1.
\end{multline}
Based on the above equation, we obtain an expression for $x_{\sigma_1}(t)$ as follows:
\begin{multline*}
x_{\sigma_1}(t) = \int\limits_0^{t-\sigma_1}e^{A\sigma_2}Bu_{\sigma_1,\sigma_2}(t)d\sigma_2 +\sum_{k=1}^m \int\limits_0^{t-\sigma_1}e^{A\sigma_2}N_kx_{\sigma_1,\sigma_2}(t)u^{(k)}_{\sigma_1,\sigma_2}(t)d\sigma_2 \\
 + \int\limits_0^{t-\sigma_1}e^{A\sigma_2}H\left(x_{\sigma_1,\sigma_2}(t)\otimes x_{\sigma_1,\sigma_2}(t)\right)d\sigma_2
\end{multline*}
and substitute it in~\eqref{eq:first_int} to have
\begin{equation*}
\begin{aligned}
x(t) &= \int\limits_{0}^t {e^{A\sigma_1}B}u_{\sigma_1}(t)d\sigma_1 +\sum_{k=1}^m\int\limits_{0}^t\int\limits_{0}^{t-\sigma_1} {e^{A\sigma_1}N_ke^{A\sigma_2} B} u^{(k)}_{\sigma_1}(t)u^{}_{\sigma_1,\sigma_2}(t)d\sigma_1d\sigma_2 \\
 &\quad+ \int\limits_{0}^t\int\limits_{0}^{t-\sigma_1}\int\limits_{0}^{t-\sigma_1} {e^{A\sigma_1} H (e^{A\sigma_2}B\otimes e^{A\sigma_3}B)}\left(u_{\sigma_1,\sigma_2}(t) \otimes u_{\sigma_1,\sigma_3}(t)\right)d\sigma_1d\sigma_2d\sigma_3 
  +\cdots.
 \end{aligned}
\end{equation*}
Repeating this process by repeatedly substituting  for the state yields the Volterra series for the QB system~\cite{sastry2013nonlinear}.  Having carefully  analyzed the \emph{kernels} of the Volterra series for the system, we define the reachability  mapping $\bar{P}$  as follows:
\begin{equation}\label{eq:Cont_Mapping}
 \bar{P} = [\bar P_1,~ \bar P_2,~\bar P_3,\dots],
\end{equation}
where the $\bar P_i$'s are:
\begin{equation}\label{eq:defined_barP}
\begin{aligned}
 \bar P_1(t_1)&= e^{At_1}B,\\
  \bar P_2(t_1,t_2)&= e^{At_2}\bbm N_1,\ldots,N_m\ebm \left(I_m\otimes \bar P_1(t_1)\right),\\
    \vdots\qquad&\qquad\qquad\vdots\\
    \bar P_i (t_1,\ldots, t_i)&= e^{At_i}\Big[H \big[\bar P_1(t_1)\otimes \bar P_{i-2}(t_{2},\ldots,t_{i-1}),\bar P_2(t_1,t_2)\otimes \bar P_{i-3}(t_{3},\ldots,t_{i-1}),\\
    &\qquad \ldots, \bar P_{i-2}(t_1,\ldots,t_{i-2})\otimes \bar P_1(t_{i-1})\big],\\
    &  \qquad \bbm N_1,\ldots,N_m\ebm\left(I_m\otimes \bar P_{i-1}(t_1,\ldots,t_{i-1})\right)\Big],\forall~i\geq 3.\\
\end{aligned}
\end{equation}
Using the mapping $\bar P$~\eqref{eq:Cont_Mapping}, we  define the reachability Gramian $P$ as
\begin{equation}\label{eq:Cont_Gram}
 P = \sum_{i=1}^{\infty}P_i\qquad \text{with} \qquad P_i = \int\limits_0^{\infty}\cdots \int\limits_0^{\infty} \bar P_i(t_1,\ldots,t_i)\bar P_i^T(t_1,\ldots,t_i)dt_1\cdots dt_i.
\end{equation}

In what follows, we show the equivalence between the above proposed reachability Gramian and the solution of a certain type of quadratic Lyapunov equation. 
\begin{theorem}\label{thm:con_gram}
 Consider the QB system~\eqref{eq:Quad_bilin_Sys} with a stable matrix $A$. If the reachability Gramian $P$ of the system defined as in~\eqref{eq:Cont_Gram} exists, then the Gramian $P$ satisfies the generalized quadratic  Lyapunov equation, given by
 \begin{equation}\label{eq:cont_lyap}
 AP + PA^T + H(P\otimes P) H^T +  \sum_{k=1}^mN_kPN_k^T + BB^T = 0.
 \end{equation}
\end{theorem}
\begin{proof}
We begin by considering the first term in the summation~\eqref{eq:Cont_Gram}. This is,
\begin{equation*}
P_1 = \int_0^\infty\bar P_1 \bar P_1^T dt_1 =  \int_0^\infty e^{At_1}BB^T e^{A^Tt_1}dt_1.
\end{equation*}
As shown, e.g., in \cite{morAnt05}, $P_1$ satisfies the following Lyapunov equation, provided $A$ is stable:
\begin{equation}\label{eq:1c}
AP_1 + P_1A^T + BB^T = 0.
\end{equation}
Next, we consider  the second term in the summation~\eqref{eq:Cont_Gram}:
\begin{align*}
P_2 &=  \int_0^\infty  \int_0^\infty \bar P_2 \bar P_2^T dt_1dt_2 \\
&=   \int\limits_0^\infty \int\limits_0^\infty e^{At_2}\bbm N_1,\ldots,N_m\ebm \left(I_m\otimes \left(e^{At_1}BB^Te^{A^Tt_1}\right)\right) \bbm N_1,\ldots N_m\ebm^T e^{A^Tt_2}dt_1dt_2\\
&=  \sum_{k=1}^m \int_0^\infty   e^{At_2}N_k \Big( \int_0^\infty e^{At_1}BB^Te^{A^Tt_1}dt_1\Big) N_k^T e^{A^Tt_2}dt_1dt_2\\
&= \sum_{k=1}^m  \int_0^\infty e^{At_2}N_kP_1 N_k^T e^{A^Tt_2}dt_2.
\end{align*}
Again using the integral representation of the solution to Lyapunov equations \cite{morAnt05}, we see that $P_2$ is the solution of the following Lyapunov equation:
\begin{equation}\label{eq:2c}
AP_2 + P_2A^T + \sum_{k=1}^mN_kP_1N_k^T = 0.
\end{equation}
Finally, we consider the $i$th term, for $i\geq 3$, which is
\begin{align*}
P_i &= \int_0^\infty\cdots \int_0^\infty \bar P_i \bar P_i^T dt_1\cdots dt_i\\
&=  \int\limits_0^\infty e^{At_i}\left[H \left[\int\limits_0^\infty\cF\left(\bar P_1(t_1)\right) dt_1\otimes\int\limits_0^\infty\cdots \int\limits_0^\infty \cF\left(\bar P_{i-2}(t_{2},\ldots,t_{i-1})\right) dt_2\cdots dt_{i-1}  \right.\right.\\
&\quad\left.\left.+\cdots + \int\limits_0^\infty\cdots \int\limits_0^\infty \cF\left(\bar P_{i-2}(t_{1},\ldots,t_{i-2}) \right) dt_1\cdots dt_{i-2}\otimes  \int\limits_0^\infty  \cF\left(\bar P_1(t_{i-1}) \right)dt_{i-1}\right]\right.H^T \\
&\quad \left.+  \sum_{k=1}^m N_k \left(\int_0^\infty\cdots \int_0^\infty \cF\left(\bar P_{i-1}(t_1,\ldots,t_{i-1})\right) \right)N_k^T\right]e^{A^Tt_i}dt_i,
\end{align*}
where  we use the shorthand  $\cF(\cA) := \cA\cA^T$. Thus, we have
\begin{align*}
P_i &=  \int_0^\infty e^{At_i}\Big[H(P_1\otimes P_{i-2} + \cdots + P_{i-2}\otimes P_1)H^T + \sum_{k=1}^mN_kP_{i-1}N_k^T\Big]e^{A^Tt_i}dt_i.
\end{align*}
Similar to $P_1$ and $P_2$, we can show that $P_i$ satisfies the following Lyapunov equation, given in terms of the preceding $P_k$, for $k = {1,\ldots,i-1}$:
\begin{equation}\label{eq:3c}
AP_i+P_iA^T + H(P_1\otimes P_{i-2} + \cdots + P_{i-2}\otimes P_1)H^T + \sum_{k=1}^mN_kP_{i-1}N_k^T = 0.
\end{equation}
To the end, adding \eqref{eq:1c}, \eqref{eq:2c} and \eqref{eq:3c} yields
\begin{equation*}
A\left.\sum_{i=1}^\infty P_i\right. + \left.\sum_{i=1}^\infty P_i\right.A^T + H\left(\sum_{i=1}^\infty P_i \otimes \sum_{i=1}^\infty P_i \right)H^T + \sum_{k=1}^m N_k\left(\sum_{i=1}^\infty P_i\right) N_k^T +BB^T = 0.
\end{equation*}
This implies that $P = \sum_{i=1}^\infty P_i$ solves the generalized quadratic  Lyapunov equation given by~\eqref{eq:cont_lyap}.
\end{proof}

\subsection{Dual system and observability Gramian for QB system}
We first derive the dual system for the QB system; the dual system plays an important role in determining the observability Gramian for the QB system~\eqref{eq:Quad_bilin_Sys}, and we aim at determining the observability Gramian in a similar fashion as done for the reachability Gramian in the preceding subsection. From linear and bilinear systems, we know that the observability Gramian of the dual system is the same as the reachability Gramian; here, we also consider the same analogy. If we compare the system~\eqref{eq:Quad_bilin_Sys} with the general nonlinear system as shown in~\eqref{eq:Gen_Nonlinear},  it turns out that for the system~\eqref{eq:Quad_bilin_Sys}
\begin{align*}
 \cA(x,u,t) &= A + H(x\otimes I) + \sum_{k=1}^mN_ku_k,\quad  \cB(x,u,t) = B~~\text{and}~~ \cC(x,u,t) = C.
\end{align*}
Using \Cref{lemma:adjointsys}, we can write down the state-space realization of the adjoint operator of the QB system as follows:
\begin{subequations}\label{eq:Adjoint_QBDAE}
 \begin{align}
  \dot{x}(t) &= Ax(t) + H(x(t) \otimes x(t))  + \sum_{k=1}^mN_kx(t)u_k(t) + Bu(t), \quad x(0) = 0,\\
  \dot{z}(t) &= -A^Tz(t) - ( x(t)^T\otimes I)H^T z(t) - \sum_{k=1}^mN_k^T z(t)u_k(t) - C^Tu_d(t), \\ &\hspace{9cm}z(\infty) = 0,\nonumber\\
  y_d(t) &= B^Tz(t),
 \end{align}
\end{subequations}
where $z(t)\in \Rn, u_d(t) \in \R$ and $y_d \in \R$ can be interpreted as the dual state, dual input and dual output vectors of the system, respectively. Next, we attempt to utilize the existing knowledge for the tensor multiplications and matricization  to simplify the term $(x(t)^T\otimes I)H^T z(t)$ in the system~\eqref{eq:Adjoint_QBDAE} and to write it in  the form of $x(t)\otimes z(t)$.

For this, we review some of the basic properties of tensor theory.  Following~\cite{koldatensor09}, the \emph{fiber} of a 3-dimensional tensor $\cH$ can be defined by fixing each index except one, e.g., $\cH(:,j,k)$,$\cH(j,:,k)$ and $\cH(j,k,:)$.  From the computational  point of view, it is advantageous to consider the matrices associated with the tensor, which can be obtained via unfolding a tensor into a matrix.    The process of unfolding a tensor into a matrix is called \emph{matricization}, and  the mode-$\mu$ matricization of the tensor $\cH$ is denoted by $\cH^{(\mu)}$. For an $l$-dimensional tensor, there are $l$ different possible ways to unfold the tensor into a matrix. We refer to \cite{morBenB15,koldatensor09} for more detailed insights into matricization. Similar to matrix multiplications, one can carry out tensor multiplication using matricization of the tensor~\cite{koldatensor09}. For instance, the mode-$\mu$ product of $\cH$ and a matrix $X\in\R^{n\times s}$ gives a tensor $\cF \in\R^{s\times n\times n}$, satisfying
\begin{equation*}
 \cF = \cH\times_\mu X\quad  \Leftrightarrow   \quad \cF^{(\mu)} = X\cH^{(\mu)}.
\end{equation*}
Analogously, if we define a tensor-matrices product as:
\begin{equation*}
 \cF = \cH \times_1 X\times_2 Y\times_3 Z,
\end{equation*}
where $\cF\in \R^{q_1\times q_2\times q_3}$, $X\in \R^{n\times q_1}$ and $Y\in \R^{n\times q_2}$ and $Z \in \R^{n\times q_3}$, then the following relations are fulfilled:
\begin{subequations}\label{eq:tensor_matricization}
\begin{align}
\cF^{(1)} &= X^T\cH^{(1)}(Y\otimes Z ), \\
\cF^{(2)} &= Z^T\cH^{(2)}(Y\otimes X ) ,\\
\cF^{(3)} &= Y^T\cH^{(3)}(Z \otimes X ) .
\end{align}
\end{subequations}

Coming back to the QB system, the matrix $H\in \R^{n\times n^2}$ in the system denotes a Hessian, which can be seen as an unfolding of a 3-dimensional tensor $\cH\in \R^{n\times n\times n}$.   Here, we choose the tensor $\cH\in \R^{n\times n\times n}$ such that its mode-1 matricization is the same as the Hessian $H$, i.e., $H = \cH^{(1)}$. Next, let us consider a tensor $\cT\in \R^{1\times n\times 1}$, whose mode-1 matricization $\cT^{(1)}$ is given by
\begin{equation*}
 \cT^{(1)}  = z(t)^T H (x{(t)\otimes I}) = z(t)^T \cH^{(1)} (x{(t)\otimes I}).
\end{equation*}
We then observe that the mode-1 matricization of the tensor $\cT$ is a transpose of the mode-2 matricization, i.e., $\cT^{(1)} = \left(\cT^{(2)}\right)^T$, leading to
\begin{equation*}
 \cT^{(1)} = \left(\cT^{(2)}\right)^T = \left(x(t)\otimes z(t)\right)^T (\cH^{(2)})^T.
\end{equation*}
Therefore, we can rewrite the system~\eqref{eq:Adjoint_QBDAE} as:
\begin{subequations}\label{eq:Adjoint_QBDAE1}
 \begin{align}
  \dot{x}(t) &= Ax(t) + H(x(t) \otimes x(t))  +  \sum_{k=1}^mN_kx(t)u_k(t) + Bu(t), \qquad~~~~~ x(0) = 0,\label{eq:state1}\\
  \dot{z}(t) &= -A^Tz(t) - \cH^{(2)} x(t)\otimes z(t) - \sum_{k=1}^m N_k^T u_k(t) z(t) - C^Tu_d(t),~ z(\infty) = 0,\\
  y_d(t) &= B^Tz(t).
 \end{align}
\end{subequations}
In the meantime, we like to point out that there are two possibilities to define $\cA(x,u,t)$ in the case of a QB system. One is  $\cA(x,u,t) = A + H(x\otimes I) + \sum_{k=1}^mN_ku_k$, which we have used in the above discussion; however, there is  another possibility to define  $\cA(x,u,t)$  as  $\tilde{\cA}(x,u,t)= A + H(I\otimes x) + \sum_{k=1}^mN_ku_k$, leading to the nonlinear Hilbert adjoint operator whose  state-space realization is given as:
\begin{subequations}\label{eq:Adjoint_QBDAE2}
	\begin{align}
	\dot{x}(t) &= Ax(t) + H(x(t) \otimes x(t))  +  \sum_{k=1}^mN_kx(t)u_k(t) + Bu(t), \qquad~~~~~ x(0) = 0,\label{eq:state1}\\
	\dot{z}(t) &= -A^Tz(t) - \cH^{(3)} x(t)\otimes z(t) - \sum_{k=1}^m N_k^T u_k(t) z(t) - C^Tu_d(t),~ z(\infty) = 0,\\
	y_d(t) &= B^Tz(t).
	\end{align}
\end{subequations}
It can be noticed that  the realizations~\eqref{eq:Adjoint_QBDAE1} and~\eqref{eq:Adjoint_QBDAE2} are the same, except the appearance of $\cH^{(2)}$ in~\eqref{eq:Adjoint_QBDAE1} instead of $\cH^{(3)}$ in~\eqref{eq:Adjoint_QBDAE2}. Nonetheless, if one assumes that the Hessian $H$ is symmetric,  i.e., $H(u\otimes v) = H(v\otimes u)$ for $u,v\in \Rn$, then the mode-2 and mode-3 matricizations coincide, i.e., $\cH^{(2)} = \cH^{(3)}$. However, the Hessian $H$, obtained after discretization of the governing equations,  may not be symmetric; but as shown in~\cite{morBenB15} the Hessian can be modified in such a way that it becomes symmetric without any change in the system dynamics. Therefore, in the rest of the paper,    without loss of generality, we assume that the Hessian $H$ is symmetric.

Now, we turn our attention towards determining the observability Gramian for the QB system by utilizing the state-space realization of the Hilbert adjoint operator (dual system). For this, we follow the same steps as used for determining the reachability Gramian. Using the dual system~\eqref{eq:Adjoint_QBDAE1}, one can write the dual state $z(t)$ of the dual system at time $t$ as follows:
 \begin{equation*}
\begin{aligned}
 z(t) &= \int_{\infty}^t e^{-A^T(t-\sigma_1)}C^T u_d(\sigma_1)d\sigma_1 +  \sum_{k=1}^m \int_{\infty}^t e^{-A^T(t-\sigma_1)}N_k^T z(\sigma_1)u_k(\sigma_1)d\sigma_1, \\
 &\quad + \int_{\infty}^t e^{-A^T(t-\sigma_1)}\cH^{(2)}\left(x(\sigma_1)\otimes z(\sigma_1)\right)d\sigma_1,
\end{aligned}
\end{equation*}
which after an appropriate change of variable leads to
\begin{equation}\label{eq:dual_newvar}
\begin{aligned}
z(t)  &= \int_{\infty}^0 e^{A^T \sigma_1}C^T u^{(d)}(t+\sigma_1)d\sigma_1 + \sum_{k=1}^m \int_{\infty}^0 e^{A^T\sigma_1}N_k^T z(t+\sigma_1)u_k(t+\sigma_1)d\sigma_1 \\
 &\quad + \int_{\infty}^0 e^{A^T\sigma_1}\cH^{(2)}\big(x(t+\sigma_1)\otimes z(t+\sigma_1)\big)d\sigma_1.
 \end{aligned}
\end{equation}
\Cref{eq:state1} gives the expression for $x(t+\sigma_1)$. This is
\begin{equation*}
\begin{aligned}
  x(t+\sigma_1) &= \int_0^{t+\sigma_1} e^{A\sigma_2}Bu(t+\sigma_1-\sigma_2)d\sigma_2 + \sum_{k=1}^m\int_0^{t+\sigma_1} \Big( e^{A\sigma_2}N_kx(t+\sigma_1-\sigma_2)  \\ 
  &   \times  u_k(t+\sigma_1-\sigma_2) \Big)d\sigma_2 + \int\limits_0^{t+\sigma_1} e^{A\sigma_2}H(x(t+\sigma_1-\sigma_2)\otimes x(t+\sigma_1-\sigma_2))d\sigma_2.
  \end{aligned}
\end{equation*}
We substitute for $x(t+\sigma_1)$ using the above equation, and $z(t+\sigma_1)$ using \eqref{eq:dual_newvar}, which gives rise to the following expression:
\begin{equation}\label{eq:sys_dual}
 \begin{aligned}
   z(t) &= \int_{\infty}^0 e^{A^T\sigma_1}C^T u_d(t+\sigma_1)d\sigma_1   +\sum_{k=1}^m  \int_{\infty}^0\int_{\infty}^0 e^{A^T\sigma_1}N_k^T \\ 
   &\quad \times e^{A^T\sigma_2}C^T u_d(t+\sigma_1+\sigma_2)u_k(t+\sigma_1)d\sigma_1d\sigma_2 + \int_{\infty}^0\int_0^{t+\sigma_1}\int_{\infty}^{0} e^{A^T\sigma_1} \\ 
   &\quad \times \cH^{(2)}\Big( e^{A\sigma_2}B\otimes e^{A^T\sigma_3}C^T \Big) u(t+\sigma_1-\sigma_2)u_d(t+\sigma_1+\sigma_3)d\sigma_1d\sigma_2d\sigma_3  + \cdots.
 \end{aligned}
\end{equation}
By repeatedly substituting  for the state $x$ and the dual state $z$, we derive the Volterra series for the dual system,
although the notation becomes much more complicated.  Carefully inspecting the kernels of the Volterra series of the dual system, we define the observability mapping $\bar Q$, similar to the  reachability mapping, as follows:
\begin{equation}\label{eq:obser_mapping}
 \bar{Q} = [ \bar Q_1,~\bar Q_2,~ \bar Q_3,\ldots],
\end{equation}
in which
\begin{align*}
 \bar Q_1(t_1) &=  e^{A^Tt_1}C^T, \\
 \bar Q_2(t_1,t_2) &=  e^{A^Tt_2}\bbm N^T_1,\cdots N_m^T\ebm \left(I_m\otimes \bar Q_1(t_1)\right),\\
 \vdots\qquad&\qquad\quad\vdots\nonumber\\
  \bar Q_i (t_1,\ldots, t_i)&= e^{A^Tt_i}\Big[\cH^{(2)} \big[\bar P_1(t_1)\otimes \bar Q_{i-2}(t_{2},\ldots,t_{i-1}), \nonumber \\ 
     &\qquad \qquad \ldots, \bar P_{i-2}(t_{1},\ldots,t_{i-2})\otimes \bar Q_1(t_{i-1})\big], \\
  &\qquad \qquad\bbm N^T_1,\ldots,N^T_m\ebm\left(I_m\otimes \bar Q_{i-1}(t_1,\ldots,t_{i-1})\right)\Big],\forall~i\geq 3.
\end{align*}
where $\bar P_i(t_1,\ldots,t_i)$ are defined in \eqref{eq:defined_barP}. Based on the above observability mapping, we define the observability Gramian $Q$ of the QB system as
\begin{equation}\label{eq:Obser_Gram}
 Q = \sum_{i=1}^{\infty}Q_i\quad \mbox{with}\quad Q_i = \int_0^\infty\cdots \int_0^\infty \bar Q_i\bar Q_i^Tdt_1\cdots dt_i.
\end{equation}
Analogous to the reachability Gramian, we next show a relation between the observability Gramian and the solution of  a generalized  Lyapunov equation.
\begin{theorem}\label{thm:obser_gram}
 Consider the QB system~\eqref{eq:Quad_bilin_Sys}  with a stable matrix $A$, and let $Q$, defined in~\eqref{eq:Obser_Gram},  be the observability Gramian of the system and assume it exists.  Then, the Gramian $Q$ satisfies the following  Lyapunov equation:
 \begin{equation}\label{eq:obser_lyap}
 A^TQ + QA  +   \cH^{(2)}(P\otimes Q) (\cH^{(2)})^T + \sum_{k=1}^m N_k^TQN_k +C^TC = 0,
 \end{equation}
 where $P$ is the reachability Gramian of the system, i.e., the solution of the generalized quadratic Lyapunov equation~\eqref{eq:Cont_Gram}.
\end{theorem}
\begin{proof}
The proof of the above theorem is analogous to the proof of \Cref{thm:con_gram}; therefore, we skip it for the brevity of the paper.
\end{proof}

\begin{remark}
As one would expect, the Gramians for QB systems reduce to the Gramians for bilinear systems~\cite{morBenD11} if the quadratic term is zero, i.e., $H = 0$.
\end{remark}

Furthermore, it will also be  interesting to look at a truncated version of the Gramians of the QB system based on the leading kernels of the Volterra series. We call a truncated version of the Gramians  \emph{truncated} Gramians of QB systems.  For this, let us consider  approximate reachability and observability mappings as follows:
\begin{equation*}
\tilde P_\cT = \bbm \tilde P_1,\tilde  P_2,\tilde P_3\ebm,\qquad \tilde Q_\cT = \bbm \tilde Q_1, \tilde Q_2,\tilde Q_3\ebm,
\end{equation*}
where
\begin{align*}
 \tP_1(t_1) &=  e^{At_1}B, & \tQ_1(t_1)&= e^{A^Tt_1}C^T,\\
\tP_2(t_1,t_2) &=  e^{At_2}\bbm N_1,\ldots,N_m\ebm \left(I_m\otimes \tP_1(t_1)\right), \\ 
 \tQ_2(t_1,t_2)&= e^{A^Tt_2}\bbm N_1^T,\ldots,N_m^T\ebm \left(I_m\otimes \tQ_1(t_1)\right),\\
\tP_3(t_1,t_2,t_3)&= e^{At_3}H (\tP_1(t_1)\otimes \tP_1(t_2)),\\
  \tQ_3(t_1,t_2,t_3)&= e^{A^Tt_3}\cH^{(2)}(\tP_1(t_1)\otimes \tQ_1(t_2)).
\end{align*}
Then, one can define the truncated reachability and observability Gramians in the similar fashion as the Gramians of the system:
\begin{subequations}\label{eq:trun_con}
	\begin{align}
P_\cT & = \sum_{i=1}^3 \hP_i,\quad\text{where}\quad \hP_i =  \int_0^\infty \tP_i(t_1,\ldots,t_i)\tP_i^T(t_1,\ldots,t_i)dt_1\cdots dt_i, \label{eq:tru_con}\\
 Q_\cT &= \sum_{i=1}^3\hQ_i,\quad \text{where}\quad \hQ_i = \int_0^\infty \tQ_i(t_1,\ldots,t_i)\tQ_i^T(t_1,\ldots,t_i)dt_1\cdots dt_i,\label{eq:tru_obs}
\end{align}
\end{subequations}
respectively. Similar to the Gramians $P$ and $Q$,  in the following we derive the relation between these truncated Gramians and the solutions of the Lyapunov equations.
\begin{corollary}\label{coro:tru_gram}
Let  $P_\cT$ and $Q_\cT$ be the truncated Gramians of the QB system as defined in~\eqref{eq:trun_con}. Then, $P_\cT$ and $Q_\cT$ satisfy the following Lyapunov equations:
\begin{subequations}\label{eq:tru_gramians}
\begin{align}
AP_\cT+P_\cT A^T + H(\hP_1\otimes \hP_1)H^T  + \sum_{k=1}^mN_k\hP_1N_k^T + BB^T  &=0,\quad \mbox{and}\\
A^TQ_\cT+Q_\cT A +  \cH^{(2)}(\hP_1\otimes \hQ_1)(\cH^{(2)})^T+ \sum_{k=1}^m N_k^T\hQ_1N_k + C^TC &=0, \label{eq:tru_obs_Gram}
\end{align}
\end{subequations}
respectively, where $P_1$ and $Q_1$ are  solutions to the following Lyapunov equations:
\begin{align}
A\hP_1+\hP_1 A^T + BB^T &=0,\quad\mbox{and}\label{eq:linear_CG}\\
A^T\hQ_1+\hQ_1 A + C^TC &=0,\quad\text{respectively.}\label{eq:linear_OG}
\end{align}
\end{corollary}
	\begin{proof}
		We begin by showing the relation between the truncated reachability Gramian $P_\cT$ and the solution of the Lyapunov equation. First, note that the first two terms of the reachability Gramian $P$~\eqref{eq:tru_con} and the truncated reachability Gramian $P_\cT$~\eqref{eq:Cont_Gram} are the same, i.e., $\hP_1 = P_1$ and $\hP_2 = P_2$, and $\hP_1$ and $\hP_2$ are the unique solutions of the following Lyapunov equations for a stable matrix $A$:
		\begin{align}
		A\hP_1 + \hP_1A^T +BB^T &=0, \quad\text{and} \label{eq:tru_P1}\\
		A\hP_2 + \hP_2A^T + \sum_{k=1}^mN_k\hP_1N_k^T = 0.\label{eq:tru_P2}
		\end{align}
		Now, we consider the third term in the summation~\eqref{eq:tru_con}. This is 
		\begin{align*}
		 P_3 &= \int_0^\infty\int_0^\infty\int_0^\infty \tP_3(t_1,t_2,t_3)\tP_3^T(t_1,t_2,t_3) dt_1dt_2dt_3\\
		 &= \int_0^\infty\int_0^\infty\int_0^\infty e^{At_3}H (\tP_1(t_1)\tP^T(t_1)\otimes \tP_1(t_2)\tP^T(t_2))H^T e^{A^Tt_3}dt_1dt_2dt_3 \\
 		 &= \int_0^\infty e^{At_3}H \left( \left(\int_0^\infty \tP_1(t_1)\tP^T(t_1)dt_1\right)\otimes \left(\int_0^\infty\tP_1(t_2)\tP^T(t_2)dt_2\right)\right)H^T e^{A^Tt_3}dt_3 \\
 		 &= \int_0^\infty e^{At_3}H \left( \hP_1 \otimes \hP_1\right)H^T e^{A^Tt_3}dt_3. 
		 \end{align*}
		 Furthermore, we use the relation between the above integral representation and the solution of Lyapunov equation to show that $\hP_3$ solves:
		 \begin{equation}\label{eq:tru_P3}
		 A\hP_3 + \hP_3A^T + H(\hP_1\otimes \hP_1)H^T =0 .
		 \end{equation}
		 Summing \eqref{eq:tru_P1}, \eqref{eq:tru_P2} and \eqref{eq:tru_P3} yields
		 \begin{equation}
		 AP_\cT + \cP_{\cT}A^T + H(\hP_1\otimes \hP_1)  + \sum_{k=1}^mN_k\hP_1N_k + BB^T=0.
		 \end{equation}
		 Analogously, we can show that $Q_\cT$ solves \eqref{eq:tru_obs_Gram}, thus concluding the proof.
	\end{proof}
	
	We will investigate the advantages of these truncated Gramians in the model reduction framework  in the later part of the paper.
	
 Next, we study the connection between the proposed Gramians for the QB system  and energy functionals. Also, we show how the definiteness of the Gramians is related to reachability and observability of the QB systems. These all suggest us how to determine the state components that are hard to control as well as hard to observe.
\section{Energy  Functionals and MOR for QB systems} \label{sec:energyfunctionals}We start by establishing the conditions under which  the Gramians approximate the energy functionals of the QB system, in the quadratic forms.
\subsection{Comparison of energy functionals with  Gramians}\label{subsec:energy}
By using \Cref{thm:energy_function}, we obtain the following nonlinear partial differential equation, whose solution gives the controllability energy functional for the QB system:
\begin{equation}\label{eq:control_energy_QB}
\begin{aligned}
 &\dfrac{\partial L_c}{\partial x}(Ax + H~x\otimes x) +  (Ax + H~x\otimes x)^T\dfrac{\partial L_c}{\partial x}^T  \\
 &\qquad + \dfrac{\partial L_c}{\partial x}\left(\bbm N_1,\ldots, N_m \ebm (I_m\otimes x) + B\right)\left(\bbm N_1,\ldots, N_m \ebm (I_m\otimes x) + B\right)^T \dfrac{\partial L_c}{\partial x} ^T = 0.
\end{aligned}
\end{equation}
 Unlike in the case of linear systems, the controllability energy functional $L_c(x)$ for nonlinear systems cannot be expressed as a simple quadratic form, i.e., $L_c(x) = x^T\tP^{-1}x$, where $\tP$ is a constant matrix. 
 
 For nonlinear systems, the energy functionals are rather complicated nonlinear functions, depending on the state vector. Thus, we aim at providing some bounds between the quadratic form of the proposed Gramians for QB systems and energy functionals. For the controllability energy functional, we  extend the reasoning given in \cite{morBenD11,bennertruncated} for bilinear systems.
 
\begin{theorem}\label{thm:con_bound}
Consider a controllable QB system~\eqref{eq:Quad_bilin_Sys}  with a stable matrix $A$. Let $P>0$ be  its reachability Gramian which is  the unique definite solution of the quadratic Lyapunov equation~\eqref{eq:cont_lyap}, and $L_c(x)$ denote the controllability energy functional of the QB system, solving~\eqref{eq:control_energy_QB}.  Then, there exists a neighborhood $W$ of $0$ such that
	  \begin{equation*}
	  L_c(x) \geq \dfrac{1}{2}x^T P^{-1}x, ~\mbox{where}~x\in W(0).
	  \end{equation*}
\end{theorem}
\begin{proof}
	Consider a state $x_0$ and let a control input $u=u_{0}:(-\infty,0]\rightarrow \Rm$, which minimizes the input energy in the definition of $L_o(x_0)$ and steers the system from $0$ to $x_0$. Now, we consider the time-varying homogeneous nonlinear differential equation
	\begin{equation}\label{eq:phi_TV}
	\dot \phi = \left(A + H(\phi\otimes I) +\sum_{k=1}^mN_ku_k(t) \right) \phi =: A_{u}\phi(t),
	\end{equation}
 and its fundamental solution $\Phi_{u}(t,\tau)$. The system~\eqref{eq:phi_TV} can thus be interpreted as a time-varying system. The reachability Gramian of the time-varying control system~\cite{shokoohi1983linear,verriest1983generalized} $\dot x = A_{u}x(t) + Bu(t)$ is  given by
\begin{equation*}
 P_u = \int_{-\infty}^0\Phi(0,\tau)BB^T\Phi(0,\tau)^Td\tau.
\end{equation*}
The input $u$ also steers the time-varying system from $0$ to $x_0$; therefore, we have
$$\|u\|_{L_2}^2 \geq \dfrac{1}{2}x^TP_u^{-1}x.$$
An alternative way to determine $P_u$ can be given by
\begin{equation*}
 P_u = \int^{\infty}_0\tilde\Phi(t,0)^TBB^T\tilde\Phi(t,0)dt,
\end{equation*}
where $\tilde\Phi$ is the fundamental solution of the following differential equation
 \begin{equation}\label{eq:tildephi_TV}
  \dot {\tilde\Phi} = \left(A^T + \cH^{(2)}(x(-t)\otimes I) +\sum_{k=1}^mN_k^Tu_k(-t) \right) \tilde\Phi~~~ \mbox{with}~~~\Phi(t,t) = I,
 \end{equation}
 and $x(t)$ is the solution of
 $$\dot x(t) = Ax(t) + H(x\otimes x) + \sum_{k=1}^mN_kx(t)u_k(t) + Bu(t).$$
Then, we define $\eta(t)$, satisfying $\eta(t) = \tilde\Phi(t,0)x_0$.  Since it is assumed that the QB system is controllable,  the state $x_0$ can be reached by using a finite input energy, i.e., $\|u\|_{L_2} <\infty$. Hence, the input $u(t)$ is a  square-integrable function over $t \in (-\infty,0]$ and so is $x(t)$. This implies that $\lim\limits_{t\rightarrow\infty}\eta(t) \rightarrow 0$, provided $A$ is stable. Thus, we have
\begin{align*}
 x_0^TPx_0 &= -\int_{0}^\infty \dfrac{d}{dt} \left(\eta(t)^TP\eta(t)\right)dt \allowdisplaybreaks\\
 & =-\int_0^\infty  \eta(t)^T\left( \left(A + H(x(-t)\otimes I) +\sum_{k=1}^mN_ku_k(-t) \right)P  \right. \\
 & \qquad \left. + P\left( A^T + \cH^{(2)}(x(-t)\otimes I) +\sum_{k=1}^mN_k^Tu_k(-t)\right) \right)\eta(t)dt \allowdisplaybreaks \\
  & =-\int_0^\infty  \eta(t)^T \left(AP + PA^T + H(P\otimes P)H^T + \sum_{k=1}^mN_kPN_k^T \right) \eta(t)  \\
  & \qquad+ \Bigg( H(P\otimes P)H^T - H(x(-t)\otimes I)P -P \cH^{(2)}(x(-t)\otimes I)   \\
 &\qquad    +  \sum_{k=1}^m\left(N_kPN_k -PN_k^Tu_k(-t) -N_k^TPu_k(-t) \right) \Bigg)\eta(t)dt.
 \end{align*}
Now, we have
\begin{multline*}
  -\int_0^\infty  \eta(t)^T \left(AP + PA^T + H(P\otimes P)H^T + \sum_{k=1}^mN_kPN_k^T \right) \eta(t) \\ =\int_0^\infty  \eta(t)^T  BB^T 
   \eta(t) = x_0^TP_ux_0.
\end{multline*}
Hence, if
\begin{multline}\label{eq:HNP_rel}
 \int_0^\infty \eta(t)^T \bigg( H(P\otimes P)H^T - H(x(-t)\otimes I)P -P \cH^{(2)}(x(-t)\otimes I)   \\
   +  \sum_{k=1}^m\left(N_kPN_k -PN_k^Tu_k(-t) -N_k^TPu_k(-t) \right) \bigg)\eta(t)dt \geq 0,
\end{multline}
then $x_0^TPx_0 \geq x_0^T P_ux_0$.
Further, if  $x_0$ lies in a small ball   $W$ in the neighborhood of  the origin, i.e., $x_0\in W(0)$, then a small input $u$  is sufficient to  steer the system from $0$ to $x_0$ and  $x(t) \in W(0) $ for $t \in (-\infty,0]$ which ensures that the relation~\eqref{eq:HNP_rel} holds for all $x_0 \in W(0)$.  Therefore, we have $x_0^TP^{-1}x_0 \leq x_0^T P_u^{-1}x_0$ if $x_0 \in W (0) $.
\end{proof}
Similarly,  we next show an upper bound for the observability energy functional for the QB system in terms of the observability Gramian (in the quadratic form).
\begin{theorem}\label{thm:obs:bound}
 Consider the QB system~\eqref{eq:Quad_bilin_Sys} with $B\equiv 0 $ and an initial condition $x_0$, and let  $L_o$ be the observability energy functional.  Let  $P>0$ and $Q\geq 0$ be  solutions to the generalized Lyapunov equations~\eqref{eq:cont_lyap} and~\eqref{eq:obser_lyap}, respectively. Then, there exists a neighborhood $\tW$ of the origin such that
 \begin{equation*}
  L_o(x_0) \leq  \dfrac{1}{2}x^TQx,\quad \mbox{where}\quad x\in\tW(0).
 \end{equation*}
\end{theorem}
\begin{proof}
Using the definition of the observability energy functional, see~\Cref{def:obser3}, we have
\begin{equation}
 L_o(x_0) = \underset{\begin{subarray}{c}
 	u\in \cB_{(\alpha,\beta)} \\[5pt]
 	x(0) = x_0,x(\infty)=0
 	\end{subarray}  }{\max} \dfrac{1}{2}\int_0^\infty \tL_o(x_0,u)dt,
\end{equation}
 where $\cB_{(\alpha,\beta)} \overset{\mathrm{def}}{=} \{u \in L_2^m[0,\infty), \|u\|_{L_2}\leq \alpha, \|u\|_{L_\infty}\leq \beta \}$ and $\tL_o(x_0,u) := \|y(t)\|_2$. Thus, we have
\begin{align*}
 \tL_o(x_0,u) &= \|y(t)\|_2 = \|Cx(t)\|_2 =  x(t)^TC^TCx(t).
\end{align*}
Substituting for $C^TC$ from~\eqref{eq:obser_lyap}, we obtain
\begin{flalign*}
 \tL_o(x_0,u)&=  -2x(t)^TQAx(t) -  x(t)^T\cH^{(2)} P\otimes Q \big(\cH^{(2)}\big)^T x(t) - \sum_{k=1}^mx(t)^TN_k^TQN_kx(t) .
 \end{flalign*}
Next, we substitute for $Ax$ from~\eqref{eq:Quad_bilin_Sys} (with $B = 0$) to have
\begin{flalign*}
 \tL_o(x_0,u)&= -2x(t)^TQ\dot{x}(t)  + 2x(t)^TQHx(t)\otimes x(t) + 2\sum_{k=1}^mx(t)^TQN_kx(t)u_k(t)\\
 &\qquad -x(t)^T\cH^{(2)} \left(P\otimes Q\right) \big(\cH^{(2)}\big)^T x(t)  - \sum_{k=1}^mx(t)^TN_k^TQN_kx(t) \displaybreak\\
 &=  - \dfrac{d}{dt}(x(t)^TQx(t)) +  x(t)^T\Big( QH (I\otimes x(t)) + QH (x(t)\otimes I) \\
 &\quad + \sum_{k=1}^m(QN_k+ N^T_k Q)u_k(t)-\cH^{(2)} (P\otimes Q)  \big(\cH^{(2)}\big)^T- \sum_{k=1}^mN_k^TQN_k \Big) x(t).
 \end{flalign*}
 This gives
 \begin{align*}
 L_o(x_0) &= \underset{\begin{subarray}{c}
 	u\in \cB_{(\alpha,\beta)} \\[5pt]
 	x(0) = x_0,x(\infty)=0
 	\end{subarray}  }{\max} \dfrac{1}{2}\int_0^\infty \tL_o(x_0,u)dt,\\
  &= \dfrac{1}{2}x_0^TQx_0  +  \underset{\begin{subarray}{c}
  	u\in \cB_{(\alpha,\beta)} \\[5pt]
  	x(0) = x_0,x(\infty)=0
  	\end{subarray}  }{\max}\dfrac{1}{2}\int_0^\infty x(t)^T\left(R_H(x,u) + \sum_{k=1}^mR_{N_k}(x,u)\right) x(t)dt,
\end{align*}
where
\begin{align*}
R_H(x,u) &:=   QH (I\otimes x) + QH (x\otimes I) - \cH^{(2)} (P\otimes Q)  \big(\cH^{2}\big)^T, \\
 R_{N_k}(x,u) &:=  \left(QN_ku_k + N_k^T Qu_k- N_k^TQN_k\right).
\end{align*}
First, note that if for a vector $v$, $v^TN_k^TQN_kv = 0$, then $QN_kv = 0$. Therefore, there exist inputs $u$ for which $\|u\|_{L_\infty}$ is small, ensuring $R_{N_k}(x,u)$ is a negative semidefinite. Similarly, if for a vector $w$, $w^T\cH^{(2)} (P\otimes Q)  \big(\cH^{2}\big)^T w = 0$ and $P>0$, then $ (I\otimes Q)  \big(\cH^{2}\big)^Tw= 0$. Using \eqref{eq:tensor_matricization}, it can be shown that $QH(w\otimes I) = QH(I\otimes w) = 0$.  Now, we consider an initial condition $x_0$ lies in the small neighborhood of the origin and $u \in \cB_{(\alpha,\beta)}$  ensuring that the resulting trajectory $x(t)$ for all time $t$ is such that $R_H(x,u)$ is a negative semi-definite. Finally, we get
 \begin{align*}
 L_o(x_0)- \dfrac{1}{2}x_0^TQx_0 &\leq 0,
 \end{align*}
for $x_0$ lies in the neighborhood of the origin and for the inputs $u$, having small $L_2$ and $L_\infty$ norms and $x_0 \in \tW(0)$ This concludes the proof.
\end{proof}

Until this point, we have proven that in the neighborhood of the origin, the energy functionals of the QB system can be  approximated by the Gramians in the quadratic form. However, one can also prove similar bounds for the energy functionals using the truncated Gramians for QB systems (defined in \Cref{coro:tru_gram}). We summarize this in the following corollary.
\begin{corollary}\label{coro:tru_energy}
Consider the system~\eqref{eq:Quad_bilin_Sys}, having a stable matrix $A$, to be locally reachable and observable. Let $L_c(x)$ and $L_o(x)$  be controllability and observability energy functionals of the system, respectively, and the truncated Gramians $P_\cT>0$ and $Q_\cT >0$ be solutions to the Lyapunov equations as shown in \Cref{coro:tru_gram}.  Then, \
\begin{enumerate}[label=(\roman*)]
	\item there exists a neighborhood $W_\cT$ of the origin such that
	\begin{equation*}
	L_c(x) \geq \frac{1}{2}x^T P_\cT^{-1}x,~\mbox{where}~x\in W_\cT(0).
	\end{equation*}
	\item Moreover,  there also exists a neighborhood $\tW_\cT$ of the origin, where
	\begin{equation*}
	L_o(x) \leq  \frac{1}{2}x^T Q_\cT x,~\mbox{where}~x\in \tW_\cT(0).
	\end{equation*}
\end{enumerate}
  \end{corollary}
In what follows, we illustrate the above bounds using Gramians and truncated Gramians by considering a scalar dynamical system, where $A,H,N,B,C$ are scalars, and are denoted by $a,h,n,b,c$, respectively.
\begin{example}
 Consider a scalar system $(a,h,n,b,c)$, where $a<0$ (stability) and nonzero $h,b,c$. For simplicity, we take $n = 0$ so that we can easily obtain analytic expressions for the controllability  and observability energy functionals, denoted by $L_c(x)$ and $L_o(x)$, respectively. Assume that the system is reachable on $\R$. Then, $L_c(x)$ and $L_o(x)$ can be  determined via solving partial differential equations~\eqref{eq:Obser_Diff} and~\eqref{eq:Cont_Diff} (with $g(x) = 0$), respectively. These are:
 \begin{align*}
  L_c(x) &= - \left(ax^2  + \tfrac{2}{3}hx^3\right)\dfrac{1}{b^2},&
  L_o(x) &= -\dfrac{c^2}{2h}\left(x - \dfrac{a}{h}\log\left(\dfrac{a+hx}{a}\right)\right),
 \end{align*}
 respectively. The quadratic approximations of these energy functionals by using the Gramians, are:
\begin{align*}
\hL_c(x) &= \dfrac{x^2}{2P} \quad ~~ \text{with} \quad P = -\dfrac{-a - \sqrt{a^2 - h^2b^2}}{h^2},\\
\hL_o(x) &= \dfrac{Qx^2}{2}\quad \text{with} \quad Q = -\dfrac{c^2}{2a+h^2P},
\end{align*}
and  the approximations in terms of the truncated Gramians are:
\begin{align*}
\hL^{\scriptscriptstyle(\cT)}_c(x) &= \dfrac{x^2}{2P_\cT}\quad ~~ \text{with} \quad P_\cT = -\dfrac{h^2b^4+4a^2b^2}{8a^3},\\
\hL^{ {\scriptscriptstyle{(\cT)}} }_o(x) &= \dfrac{Q_\cT x^2}{2}\quad \text{with} \quad Q_\cT  = -\dfrac{h^2b^2c^2+4a^2c^2}{8a^3}.
\end{align*}
In order to compare these functionals, we set $a =-2,b=c=2$ and $h =1$ and plot the resulting energy functionals  in \Cref{fig:comparison_gram}.

\begin{figure}[!tb]
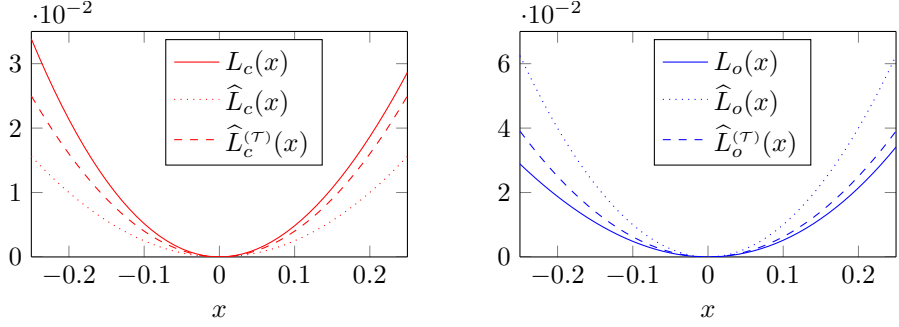

 \begin{subfigure}[h]{0.49\textwidth}
\centering
	\setlength\fheight{3cm}
	\setlength\fwidth{5.cm}
	\includetikz{Energy_QBDAE_cont}
	\caption{Comparison of the controllability energy functional  and its approximations.}
   \end{subfigure}
    \begin{subfigure}[h]{0.49\textwidth}
\centering
	\setlength\fheight{3cm}
	\setlength\fwidth{5.0cm}
	\includetikz{Energy_QBDAE_obser}
	\caption{Comparison of the observability energy functional and its approximations.}
   \end{subfigure}
   \caption{Comparison of exact energy functionals with approximated energy functionals via the Gramians and truncated Gramians.}
   \label{fig:comparison_gram}
\end{figure}

Clearly, \Cref{fig:comparison_gram} illustrates the lower and upper bounds for the controllability and observability energy functionals, respectively at least locally. Moreover, we observe that the bounds for the energy functionals, given in terms of truncated Gramians are  closer to the actual energy functionals of the system in the small neighborhood of the origin. 
 \end{example}

So far, we have shown the bounds for the energy functionals in terms of the Gramians of the QB system. In order to prove those bounds, it is assumed that $P$ is a  positive definite. However, this assumption might not be fulfilled for many QB systems, especially arising from  semi-discretization of nonlinear PDEs. Therefore, our next objective is to  provide another interpretation of the proposed Gramians and truncated Gramians, that is, the connection of Gramians and truncated Gramians with reachability and observability of the system. 
 For the observability energy functional, we consider the output $y$ of the following \emph{homogeneous}  QB system:
\begin{equation}\label{eq:homo_qbdae}
\begin{aligned}
\dot{x}(t)&= Ax + Hx(t)\otimes x(t) + \sum_{k=1}^mN_kx(t)u_k(t),\\
y(t) &= Cx(t),\qquad x(0) = x_0,
\end{aligned}
\end{equation}
as considered for bilinear systems in~\cite{morBenD11,enefungray98}. However, it might also be possible to consider an \emph{inhomogeneous} system by setting the control input $u$ completely zero, as shown in~\cite{morSch93}. We first investigate how the proposed Gramians are related to reachability and observability of the QB systems, analogues to derivation for bilinear systems  in \cite{morBenD11}.
\def\im {\ensuremath{{\mathsf{Im}}}}
\def\ker {\ensuremath{{\mathsf{Ker}}}}
\def\rank {\ensuremath{{\mathsf{rank}}}}
\begin{theorem}\label{thm:kerPQ_dyn}
	\
	\begin{enumerate}[label=(\alph*)]
		\item Consider the QB system~\eqref{eq:Quad_bilin_Sys}, and assume the reachability Gramian $P$ to be the solution of~\eqref{eq:cont_lyap}. If the system is steered from $0$ to $x_0$, where $x_0 \not\in \im P$, then $L_c(x_0) = \infty$ for all input functions $u$.
		\item Furthermore,  consider the \emph{homogeneous} QB system~\eqref{eq:homo_qbdae} and assume $P>0$ and $Q$ to be the reachability and observability Gramians of the QB system which are  solutions of~\eqref{eq:cont_lyap} and~\eqref{eq:obser_lyap}, respectively. If the initial state satisfies $x_0 \in  \ker Q$, then $L_o(x_0) = 0$.
	\end{enumerate}
\end{theorem}
\begin{proof}\
\begin{enumerate}[label=(\alph*)]
	\item By assumption, $P$ satisfies
\begin{equation}\label{eq:CG}
AP + PA^T + H(P\otimes P)H^T +\sum_{k=1}^mN_kPN_k^T + BB^T =0.
\end{equation}
Next, we consider a vector  $v \in \ker P$ and multiply the above equation from the left and right with $v^T$ and $v$, respectively to obtain
\begin{align*}
0&= v^TAPv + v^TPA^Tv + v^TH(P\otimes P)H^Tv +\sum_{k=1}^mv^TN_kPN_k^Tv + v^TBB^Tv\\
&=v^TH(P\otimes P)H^Tv +\sum_{k=1}^mv^TN_kPN_k^Tv + v^TBB^Tv.
\end{align*}
This implies $B^Tv = 0$, $PN_k^T v=0$ and $(P\otimes P)H^T v=0$. From~\eqref{eq:CG}, we  thus obtain $PA^Tv = 0$. Now we consider an arbitrary state vector $x(t)$, which is the solution of \eqref{eq:Quad_bilin_Sys} at time $t$ for any given input function $u$. If $x(t) \in \im P $ for some $t$, then we have
\begin{equation*}
\dot{x}(t)^Tv = x(t)^TA^Tv + \left(x(t)\otimes x(t)\right)^TH^Tv + \sum_{k=1}^mu_k(t)x(t)^TN_k^Tv + u(t)B^Tv = 0.
\end{equation*}
The above relation indicates that $\dot x(t) \perp v$ if $v\in \ker P$ and $x(t) \in \im P$.  It shows that \im P is invariant under the dynamics of the system. Since the initial condition $0$ lies in $\im P$,  $x(t) \in \im P$ for all $t\geq 0$. This reveals that if the final state $x_0 \not\in \im P$, then it cannot be reached from $0$; hence, $L_c(x_0) =\infty$.
\item Following the above discussion, we can show that $(I\otimes Q )  \left(\cH^{(2)}\right)^T \ker Q = 0$, $QN_k\ker Q =0$, $QA\ker Q = 0$, and $C\ker Q =0$.  Let $x(t)$ denote the solution of the homogeneous system at time $t$. If $x(t)\in \ker Q$ and a vector $\tv \in \im Q$, then we have
\begin{align*}
\tv^T\dot x(t)&= \underbrace{\tv Ax(t) }_{=0}+ \tv^TH(x(t)\otimes x(t)))+ \sum_{k=1}^m\underbrace{\tv^TN_kx(t)u_k(t)}_{=0}\\
&= x(t)^T\cH^{(2)}(x(t)\otimes \tv) = \underbrace{x(t)^T\cH^{(2)}(I\otimes \tv)}_{=0}x(t) = 0.
\end{align*}
This implies that if $x(t) \in  \ker Q$, then $\dot x(t)  \in\ker Q$.  Therefore, if the initial condition $x_0 \in   \ker Q$, then $x(t)\in \ker Q$ for all $t\geq 0$, resulting in $y(t)=C\underbrace{x(t)}_{\in \ker Q} = 0$; hence, $L_o(x_0) = 0$.
\end{enumerate}
\end{proof}

The above theorem suggests that the state components, belonging to \ker P or \ker Q, do not play a major role as far as the system dynamics are concerned. This shows that the states which belong to \ker P,  are uncontrollable, and similarly, the states, lying in \ker Q are unobservable once the uncontrollable states are removed. Furthermore, we have shown in \Cref{thm:obs:bound,thm:con_bound} the lower and upper bounds for the controllability and observability energy functions in the quadratic form of the Gramians $P$ and $Q$ of QB systems (at least in the neighborhood of the origin). This coincides with the concept of balanced truncation model reduction which aims at eliminating weakly controllable and weakly observable state components. Such states are corresponding to zero or small singular values of $P$ and $Q$. In order to find these states simultaneously, we utilize the balancing tools similar to the linear case; see, e.g.,~\cite{moral1994,morAnt05}. For this, one needs to determine the Cholesky factors of the Gramians as  $P =: S^TS$  and $Q =: R^TR$, and compute the SVD of $ SR^T =: U\Sigma V^T $, resulting in a transformation matrix $T = S^TU\Sigma^{-\tfrac{1}{2}}$. Using the matrix $T$, we obtain an equivalent QB system
\begin{equation}\label{sys:tran_QBDAE}
\begin{aligned}
\dot{\tx}(t) &= \tA \tx(t) + \tH \tx(t)\otimes \tx(t) + \sum_{k=1}^m\tN_k \tx(t)u_k(t) + \tB u(t),\\
y(t) &= \tC \tx(t),\quad \tx(0) = 0
\end{aligned}
\end{equation}
with $$\tA = T^{-1}AT,\quad\tH = T^{-1}H(T\otimes T),\quad\tN_k = T^{-1}N_kT,\quad\tB = T^{-1}B,\quad\tC = CT.$$
Then, the above transformed system~\eqref{sys:tran_QBDAE} is a balanced system, as the Gramians $\tP$ and $\tQ$ of the system~\eqref{sys:tran_QBDAE} are equal and diagonal, i.e., $\tP = \tQ = \mbox{diag}(\sigma_1,\sigma_2,\ldots,\sigma_n)$. The attractiveness of the balanced system is that it allows us to find  state components corresponding to small singular values of both $\tP$ and $\tQ$. If $\sigma_{\hn} > \sigma_{\hn+1}$, for some $\hn\in\N$, then it is easy to see that states related to $\{\sigma_{\hn+1},\ldots,\sigma_n\}$ are not only hard to control but also hard to observe; hence, they can be eliminated. In order to determine a reduced system of order $\hn$, we partition $T = \bbm T_1 &T_2\ebm$ and $T^{-1} = \bbm S_1^T & S_2^T\ebm^T$, where $T_1,S_1^T \in \R^{n\times \hn}$, and define the reduced-order system's realization as follows:
\begin{equation}\label{eq:red_realization}
\hA = S_1AT_1,~~\hH = S_1H(T_1\otimes T_1),~~\hN_k = S_1N_kT_1,~~\hB = S_1B,~~\hC = CT_1,
\end{equation}
which is generally a locally good approximate of the original system; though it is not a straightforward task to estimate the error occurring due to the truncation of the QB system unlike in the case of linear systems.   

 Based on the above discussions, we propose the following corollary, showing how the truncated Gramians of a QB system relate to reachability and observability of the system.
\begin{corollary}\label{coro:gram_trc_re}
\	\begin{enumerate}[label=(\alph*)]
		\item  Consider the QB system~\eqref{eq:Quad_bilin_Sys}, and let $P_\cT$ and $Q_\cT $ be the truncated Gramians of the system, which are solutions of the Lyapunov equations as in~\eqref{eq:tru_gramians}. If the system is steered from $0$ to $x_0$ where, $x_0 \not\in \im P_\cT$, then $L_c(x_0) = \infty$ for all input functions $u$. 
		\item Assume the QB system~\eqref{eq:Quad_bilin_Sys} is locally controllable around the origin, i.e., $(A,B)$ is controllable.  Then, for  the \emph{homogeneous} QB system~\eqref{eq:homo_qbdae}, if  the initial state $x_0 \in  \ker Q_\cT$, then $L_o(x_0) = 0$.
\end{enumerate}
\end{corollary}
The above corollary can be proven, along of the lines of the proof for \Cref{thm:kerPQ_dyn}, keeping in mind that if $\gamma~\in~\ker P_\cT$, then $\gamma$ also belongs  to $\ker P_1$, where $P_1$ is the solution to \eqref{eq:linear_CG}. Similarly, if $\xi\in \ker Q_\cT $, then $\xi$ also lies in $\ker Q_1$, where $Q_1$ is the solution to \eqref{eq:linear_OG}. This can easily be verified using simple linear algebra. Having noted this, \Cref{coro:gram_trc_re} also suggests that $\ker P_\cT$ is uncontrollable, and $\ker Q_\cT$ is also unobservable if the system is locally controllable.   Moreover, these truncated Gramians also bound the energy functions for QB systems in the quadratic form, see \Cref{coro:tru_energy}. Based on these, we conclude that the truncated Gramians are also a good candidate to use for balancing the system and to compute  the reduced-order systems.

\section{Computational Issues and Advantages of Truncated Gramians}\label{sec:computational} Up to now, we have proposed the Gramians for the QB systems and showed their relations to energy functionals of the system which allows us to determine the reduced-order systems. Here, we discuss computational issues and the advantages of considering this truncated Gramians in the MOR framework. Towards this end, we address stability issues of the reduced-order systems, obtained by using the truncated Gramians. 
\subsection{Computational issues}
One of the major concerns in applying balanced truncation MOR is that it requires the solutions of  two Lyapunov equations~\eqref{eq:cont_lyap} and \eqref{eq:obser_lyap}. These equations are quadratic in nature, which are not trivial to solve, and they appear to be computationally expensive. So far, it is not clear  how to solve these generalized quadratic  Lyapunov equation efficiently; however, under some assumptions,  a fix point iteration scheme can be employed, which is based on the theory of convergent splitting presented in~\cite{damm2001newton,schneider1965positive}. This has been studied for solving generalized Lyapunov equation for bilinear systems in~\cite{damm2008direct}, wherein the proposed stationary method is as follows:
\begin{equation}\label{eq:bilinear_iter}
\cL(X_i) = \cN(X_{i-1}) - BB^T, \qquad i = 1,2,\ldots,
\end{equation}
with $\cL(X) = AX + XA^T$ and $\cN(X_i) = -\sum_{k=1}^mN_kX_iN_k^T$. To perform this stationary iteration, we require the solution of a conventional Lyapunov equation at each iteration. Assuming $\sigma(A) \subset \C^{-}$ and spectral radius of $\cL^{-1}\cN<1$, the iteration \eqref{eq:bilinear_iter} linearly converges to a positive semidefinite solution $X$ of the generalized Lyapunov equation for bilinear systems, which is 
$$AX + XA^T + \sum_{k=1}^mN_kXN_k^T+BB^T = 0.$$
Later on, the efficiency of this iterative method was improved in~\cite{shank2014efficient} by utilizing  tools for inexact solution of $Ax = b$. The main idea was to determine a low-rank factor of $\cN(X_{i-1}) - BB^T$ by truncating the columns, corresponding to small singular values and to increase the accuracy of the low-rank solution of the linear Lyapunov equation~\eqref{eq:bilinear_iter}  with each iteration. In total, this results in an efficient method to determine a low-rank solution of the generalized Lyapunov equation for bilinear systems with the desired tolerance. For detailed insights, we refer to~\cite{shank2014efficient}.

One can utilize the same tools to determine  the solutions of generalized quadratic-type Lyapunov equations. We begin with the inexact form equation, which on convergence gives the reachability Gramian; this is,
\begin{equation}
\label{eq:quad_iter}
\cL(X_i) = \Pi(X_{i-1}) - BB^T,\quad i = 1,2,\ldots
\end{equation}
where $\cL(X) = AX + XA^T$ and $\Pi(X) = -H(X\otimes X)H^T-\sum_{k=1}^mN_kXN_k^T$. Similar to the bilinear case, if $\sigma(A) \subset \C^{-}$ and the spectral radius of $\cL^{-1}\Pi<1$, then the iteration \eqref{eq:quad_iter} converges to a positive semidefinite solution  of the generalized quadratic Lyapunov equation. Next, we determine a low-rank approximation of $\Pi(X) = -H(X\otimes X)H^T-\sum_{k=1}^mN_kXN_k^T$. For this, let us assume a low-rank product $X := FDF^T$, where $F \in \Rnk$ and a QR decomposition of $F := Q_FR_F$. We then perform an eigenvalue decomposition of the relatively small matrix $R_FDR_F^T := U\Sigma U^T$, where $\Sigma = \diag{\sigma_1,\ldots,\sigma_k}$ with $\sigma_j \geq \sigma_{j+1}$.  Assuming there exists a scalar $\beta$ such that
 $$\sqrt{\sigma_{\beta+1}^2+\cdots +\sigma_k^2} \leq \tau \sqrt{\sigma_1^2+\cdots +\sigma_k^2},$$
where $\tau$ is a given tolerance, this gives us a low-rank approximation of $X$ as:
$$X \approx \tF\tD\tF^T,$$
where $\tF = Q_F \tU$ and $\tD = \diag{\sigma_1,\ldots,\sigma_\beta}$. Following the short-hand notation, we denote $\tZ = \cT_\tau (Z)$ which gives the low-rank approximation of $ZZ^T$ with the tolerance $\tau$, i.e., $ZZ^T \approx \tZ\tZ^T$. Considering a low-rank factor of $X_{k-1} \approx Z_{k-1}Z_{k-1}^T$,  the right side of~\eqref{eq:quad_iter}
\begin{multline*}
\Pi(X_{k-1}) - BB^T \approx -[H(Z_{k-1} \otimes Z_{k-1}), \left[N_1,\ldots,N_m\right]Z_{k-1}, B ]\\
\hspace{2cm}\times[H(Z_{k-1} \otimes Z_{k-1}), \left[N_1,\ldots,N_m\right]Z_{k-1}, B ]^T
\end{multline*}
can be replaced with its truncated version $\cT_\tau (\Pi(X_{k-1}) - BB^T) =: -\F_k\F_k^T$ with the desired tolerance. This indicates that we now need to solve the following linear Lyapunov equation at each step:
\begin{equation}\label{eq:quad_ite_approx}
AX_k + X_kA = -\F_k\F_k^T,
\end{equation}
 which can be solved very efficiently by using any of the recently developed low-rank solvers for Lyapunov equations; see, e.g.,~\cite{benner2013numerical,simoncini2013computational}.  
In the following, we  outline all the necessary steps in \Cref{algo:solv_gram} to determine the Gramians by summarizing the all above discussed ingredients.
\begin{algorithm}[tb!]
 \caption{Iterative scheme to determine Gramians for QB systems.}
 \begin{algorithmic}[1]
    \Statex {\bf Input:} System matrices $ A, H,N_1,\ldots,N_m, B,C$ and tolerance $\tau$.
\Statex {\bf Output:} Low-rank factors of the Gramians: $Z_k~ (P \approx Z_kZ_k^T)$ and $ X_k~ (Q \approx X_kX_k^T)$.
    \State Solve approximately $AM + MA^T + BB^T = 0$ for $P_1 \approx Z_1Z_1^T$.
    \State Solve approximately $A^TG + GA + C^TC = 0$ for $Q_1 \approx X_1X_1^T$.
        \For {$k = 2,3,\ldots $}
        \State Determine low-rank factors:
        \Statex \label{step} \qquad\qquad $\mathbb B_k = \cT_{\tau} ([H(Z_{k-1} \otimes Z_{k-1}),  N_1Z_{k-1},\ldots,N_mZ_{k-1}  ,B ])$,
        \Statex \qquad\qquad $\mathbb C_k = \cT_{\tau} ([\cH^{(2)}(Z_{k-1} \otimes X_{k-1}),  N_1^T X_{k-1},\ldots,N_m^T X_{k-1},C^T ]) $.
        \State Solve approximately $AM + M A^T + \mathbb B_k\mathbb B_k^T = 0$ for $P_k\approx Z_k Z_k^T $.
        \State Solve approximately $A^TG + G A + \mathbb C_k \mathbb C_k^T = 0$ for $Q_k\approx X_kX_k^T $.
	\If{solutions are sufficiently accurate } stop. \EndIf
    \EndFor
\end{algorithmic}\label{algo:solv_gram}
\end{algorithm}

\begin{remark}
	At step 7 of \Cref{algo:solv_gram}, one can check the accuracy of  solutions by measuring the relative changes  in the solutions, i.e., $\dfrac{\|P_k - P_{k-1} \|}{\|P_k\|}$ and $\dfrac{\|Q_k - Q_{k-1} \|}{\|Q_k\|}$. When these relative changes are smaller than a \emph{tolerance} level, e.g. the machine precision, then one can stop the iterations to have sufficiently accurate solutions of the quadratic Lyapunov equations. 
\end{remark}

\begin{remark}
In order to employ \Cref{algo:solv_gram}, the right side of the conventional Lyapunov equation (see step \ref{step}) requires the computation of  $H(Z_i\otimes Z_i) =: \Gamma$ at each step, which is also computationally and memory-wise expensive. If $Z_i \in  \R^{n\times n_z}$, then the direct multiplication of $Z_i\otimes Z_i$ would have complexity of $\cO(n^2\cdot n_z^2)$,  leading to an unmanageable task for large-scale systems, even on  modern computer architectures. However, it is shown in~\cite{morBenB15} that $\Gamma$ can be determined  efficiently by making use of the tensor multiplication properties, which are also reported in the previous section. In the following, we provide the procedure to compute $\Gamma$ efficiently:
\begin{itemize}
\item Determine $\cY \in \R^{n_z\times n\times n}$ such that $\cY^{(2)} = Z_i^T \cH^{(2)}$.
\item Determine $\cK \in \R^{n\times n_z \times n_z}$ such that $\cK^{(3)} = Z_i^T \cY^{(3)}$.
\item Then, $\Gamma = \cK^{(1)}$.
\end{itemize}
This way, we can avoid determining the full matrix $Z_i\otimes Z_i$.
Analogously, we can also compute  the term $\cH^{(2)}(Z_i\otimes X_i)$.
\end{remark}
 
 Next, we discuss the existence of the solutions of quadratic type generalized Lyapunov equations. As noted \Cref{algo:solv_gram}, one can determine the solution of these Lyapunov equations using  fixed point iterations. In the following, we discuss sufficient conditions under which these iterations converge to finite solutions.    
\begin{theorem}
Consider a QB system as defined in~\cref{eq:Quad_bilin_Sys} and let $P$ and $Q$ be its reachability and observability Gramians, respectively.	Assume that the Gramians $P$ and $Q$ are determined using fixed point iterations as shown in \Cref{algo:solv_gram}.   Then, the Gramian $P$ converges to a positive semidefinite solution if
	\begin{enumerate}[label=(\roman*)]
	\item $A$ is stable, i.e., there exist  $0<\alpha\leq -\max(\real{\lambda_i (A)})$ and $\beta>0$  such that $\|e^{At}\| \leq \beta e^{-\alpha t}$. \label{eq:cond_Pexist1}
	\item $\dfrac{	\beta^2\Gamma_N }{2\alpha} < 1$, where $\Gamma_N := \sum_{k=1}^m\|N_k\|^2$.
	\item $ 1>\cD^2 - \dfrac{\beta^2 \Gamma_H}{\alpha}\dfrac{\beta^2\Gamma_B}{\alpha} > 0, ~~where~~ \cD:= 1-\dfrac{\beta^2\Gamma_N}{2\alpha}$,\label{eq:cond_Pexist3}
	where $\Gamma_B := \|BB^T\|$, $\Gamma_H := \|H\|^2$.
   \end{enumerate}
   and $\|P\|$ is bounded by
	\begin{equation}
	\|P\| \leq \dfrac{2\alpha}{\beta^2\Gamma_H} \left(\cD- \sqrt{\cD^2 - 4\dfrac{\beta^2 \Gamma_H}{2\alpha}\dfrac{\beta^2\Gamma_B}{2\alpha} }\right) =: \cP_\infty.
	\end{equation}
	Furthermore, the Gramian $Q$  also converges to a positive semidefinite solution if in addition to the above  conditions \ref{eq:cond_Pexist1}--\ref{eq:cond_Pexist3}, the following condition satisfies
	 		\begin{equation}
	 		\dfrac{\beta^2}{2\alpha}	\left(\Gamma_N + \tilde\Gamma_{H} \cP_\infty \right)< 1,
	 		\end{equation}
	 		where $\tilde{\Gamma}_H := \|\cH^{(2)}\|^2$. Moreover,  $\|Q\|$ is bounded by 
	 		\begin{equation}
		\|Q\| \leq \dfrac{\beta^2}{2\alpha}\Gamma_C \left(1-\dfrac{\beta^2}{2\alpha}	\left(\Gamma_N + \tilde\Gamma_H \cP_\infty \right)\right)^{-1},
	 		\end{equation}
	 		where $\Gamma_C:= \|C^TC\|.$
\end{theorem}
\begin{proof}
	Let us first consider the equation corresponding to $P_1$:
	\begin{equation}
	AP_1 + AP_1 + BB^T = 0.
	\end{equation}
	Alternatively, if $A$ is stable, we can write $P_1$ in the integral form as
	\begin{equation}
	P_1 = \int_0^\infty e^{At}BB^Te^{A^Tt}dt,
	\end{equation}
	implying  
	\begin{equation}
	\|P_1\| \leq \beta^2 \|BB^T\|  \int_0^\infty e^{-2\alpha t}dt = \dfrac{\beta^2\Gamma_B}{2\alpha},
	\end{equation}
where $\Gamma_B := \|BB^T\|$.	Next, we look at the equation corresponding to $P_k$, which is given in terms of $P_{k-1}$:
	\begin{equation}
	AP_k + P_kA^T + H(P_{k-1} \otimes P_{k-1})H^T + \sum_{k=1}^mN_kP_{k-1}N_k + BB^T=0.
	\end{equation}
	We can also write $P_k$ in an integral form, provided $A$ is stable:
	\begin{align*}
	P_k &=\int_0^\infty e^{At}\left(H(P_{k-1} \otimes P_{k-1})H^T + \sum_{k=1}^mN_kP_{k-1}N_k + BB^T\right)e^{A^Tt}dt\\
	&\leq  \beta^2\left(\Gamma_H \|P_{k-1}\|^2 + \Gamma_N\|P_{k-1}\| + \Gamma_B\right)\int_0^\infty e^{-2\alpha t}dt\\
		&\leq \beta^2  \dfrac{\left(\Gamma_H \|P_{k-1}\|^2 + \Gamma_N\|P_{k-1}\| + \Gamma_B\right)}{2\alpha},
	\end{align*}
	where $\Gamma_H := \|H\|^2$ and $\Gamma_N := \sum_{k=1}^m\|N_k\|^2$.
	If we consider an upper bound for the norm of $P_{k-1}$ in order to provide an upper bound for $P_k$ and apply \Cref{appendix_convergence}, then we know that $\lim_{k\rightarrow \infty} \|P_k\|$ is bounded if
	\begin{align*}
&1>\cD^2 - 4\dfrac{\beta^2 \Gamma_H}{2\alpha}\dfrac{\beta^2\Gamma_B}{2\alpha} \geq 0, ~~where~~ \cD:= 1-\dfrac{\beta^2\Gamma_N}{2\alpha} \quad\text{and}\quad
 \dfrac{\beta^2\Gamma_N}{2\alpha} <1,
	\end{align*}
	and $\lim_{k\rightarrow \infty} \|P_k\|$ is bounded by 
	\begin{equation*}
	\lim_{k\rightarrow \infty} \|P_k\| \leq \dfrac{2\alpha}{\beta^2\Gamma_H} \left(\cD- \sqrt{\cD^2 - 4\dfrac{\beta^2 \Gamma_H}{2\alpha}\dfrac{\beta^2\Gamma_B}{2\alpha} }\right) =: \cP_\infty.
	\end{equation*}
Now, we consider the equation corresponding to $Q_1$:
	\begin{equation*}
	A^TQ_1 + A^TQ_1 + C^TC = 0,
	\end{equation*}
which again can be rewritten as:
	\begin{equation*}
	Q_1 = \int_0^\infty e^{A^Tt}C^TCe^{At}dt
	\end{equation*}
if $A$ is stable. This	implies 
	\begin{equation*}
	\|Q_1\| \leq \beta^2\Gamma_C  \int_0^\infty e^{-2\alpha t}dt = \beta^2 \dfrac{\Gamma_C}{2\alpha},
	\end{equation*}
where $\Gamma_c := \|C^TC\|$.	Next, we look at the equation corresponding to $Q_k$, that is,
	\begin{align*}
	A^TQ_k + Q_kA + \cH^{(2)}(P_{k-1} \otimes Q_{k-1})\left(\cH^{(2)}\right)^T + \sum_{k=1}^mN_k^TQ_{k-1}N_k + C^TC&=0. 
	\end{align*}
A similar analysis for $Q_k$  yields
\begin{equation*}
\|Q_k\| \leq \dfrac{\beta^2}{2\alpha}\left( \left(\Gamma_N + \tilde\Gamma_H \|P_{k-1}\| \right)Q_{k-1} + \Gamma_C \right),
\end{equation*}
where $\tilde{\Gamma}_H := \|\cH^{(2)}\|$. Since $\|P_{k-1}\| \leq \cP_\infty$ for all $k\geq 1$, we further have
\begin{equation*}
\|Q_k\| \leq \dfrac{\beta^2}{2\alpha}\left( \left(\Gamma_N + \tilde\Gamma_H \cP_\infty \right)\|Q_{k-1}\| + \Gamma_C \right).
\end{equation*}
An additional sufficient condition under which the above recurrence formula in $\|Q_{k}\|$  converges is as follows:
		\begin{equation*}
\dfrac{\beta^2}{2\alpha}	\left(\Gamma_N + \tilde\Gamma_H \cP_\infty \right)< 1,
		\end{equation*}
		and $\lim_{k\rightarrow \infty}\|Q_k\|$ is then bounded by 
		\begin{equation*}
				\lim_{k\rightarrow \infty}\|Q_k\| \leq \dfrac{\beta^2}{2\alpha}\Gamma_C \left(1-\dfrac{\beta^2}{2\alpha}	\left(\Gamma_N + \tilde\Gamma_H \cP_\infty \right)\right)^{-1}.
		\end{equation*}
		This concludes the proof.
	\end{proof}
\begin{remark}
In \Cref{algo:solv_gram}, we propose to determine the low-rank solutions of the Lyapunov equation  at each intermediate step with the same tolerance. However, one can also consider to  increase the tolerance adaptively for computing the low-rank  solution of the Lyapunov equation  with each  iteration which probably can speed up even more, see~\cite{shank2014efficient} for the generalized Lyapunov equations for bilinear systems.  
\end{remark}

\subsection{MOR using truncated Gramians}
As noted in \Cref{sec:energyfunctionals}, the quadratic forms of both actual Gramians and its truncated versions (truncated Gramians) impose  bounds for the energy functionals of QB systems, at least in the neighborhood of the origin, and we also provide the interpretation of reachability and observability of the system in terms of Gramians and truncated Gramians.  We have seen in the previous subsection that determining Gramians $P$ and $Q$ is very challenging task for large-scale settings.  Moreover, the convergence of \Cref{algo:solv_gram}  highly depends on the spectral radius condition $\cL^{-1}\Pi$, which should be less than $1$. This condition might not be satisfied  for large $H$ and $N_k$; thus,  \Cref{algo:solv_gram} may not convergence.  On the other hand, in order to compute the truncated Gramians, there is no such convergence issue. Furthermore, it can also be  observed that the bounds for energy functionals using the truncated Gramains can be much better (in the neighborhood of the origin), for example see \Cref{fig:comparison_gram}. 

Additionally, if we remove those states that are completely uncontrollable and completely unobservable, then the truncated Gramians may provide reduced systems which are of smaller orders as compared to using the Gramians of QB systems. This is due to the fact that $P \geq P_\cT$ and $Q \geq Q_\cT$. This motivates us to use the truncated Gramians to determine the reduced-order models, and we present the square-root balanced truncation for QB systems based on these truncated Gramians in \Cref{algo:BT_qbdae}.  Furthermore, we will see in \Cref{sec:Numerical} as well that these truncated Gramians also yield very good qualitative reduced-order systems for QB systems.

\begin{algorithm}[!tb]
	\caption{Balanced truncation for QB systems (truncated version).}
	\begin{algorithmic}[1]
		\Statex {\bf Input:} System matrices $ A, H,N_k, B$ and $C$, and the order of the reduced system~$\hn$.
		\Statex {\bf Output:} The reduced system's matrices  $\hA, \hH,\hN_k,\hB, \hC.$
		\State Determine low-rank approximations of the truncated Gramians $P_\cT\approx RR^T$ and $Q_\cT~\approx~SS^T$.
		\State Compute SVD of $S ^TR$:
		\Statex \qquad $S^TR = U\Sigma V = \bbm U_1 & U_2\ebm \diag{ \Sigma_1,\Sigma_2}\bbm V_1 & V_{2}\ebm^T$,
		\Statex\qquad     where $\Sigma_1$ contains the $\hn$ largest singular values of $S^TR$.
		\State Construct the projection matrices $\cV$  and $\cW$:
		\Statex \qquad $\cV = S U_1\Sigma_1^{-\tfrac{1}{2}}$ and $\cW = R V_1\Sigma_1^{-\tfrac{1}{2}}$.
		\State Determine the reduced-order system's realization:
		\Statex \qquad$\hA = \cW^TA\cV,~~\hH =\cW^T H(\cV\otimes \cV),~~\hN_k = \cW^TN_k\cV,~~ \hB = \cW^TB,~~\hC = C\cV $.
	\end{algorithmic}\label{algo:BT_qbdae}
\end{algorithm}

\subsection{Stability Preservation} We now discuss the stability of the reduced-order systems, obtained by using \Cref{algo:BT_qbdae}. For this, we consider only the autonomous part of the QB system as follows:
\begin{equation}\label{eq:auto_QB}
\dot{x}(t) = Ax(t) + H~x(t)\otimes x(t),
\end{equation}
where $x_{eq} = 0$ is a stable equilibrium. 
{
In the following, we discuss Lyapunov stability of $x_{eq}$. For this, we first note the definition of the latter stability.
\begin{definition}
	Consider a  QB system with $u\equiv 0$ \eqref{eq:auto_QB}. If there exists a Lyapunov function $\cF : \Rn \rightarrow \R$ such that 
	\begin{equation*}
	\cF(x) >0\quad\mbox{and}\quad \dfrac{d}{dt}\cF(x)<0,\qquad \forall x\in \cB_{0,r}\backslash \{0\},
	\end{equation*}
	where $\cB_{0,r}$ is a ball of radius $r$ centered around $0$, then $x_{eq} = 0$ is a locally asymptotically stable. 
\end{definition}
}
However, many other notions of the stability of nonlinear systems are available in the literature, for instance based on a certain dissipation inequality~\cite{bond2010compact}, which might be difficult to apply in the large-scale setting. In this paper, we stick to the notion of the Lyapunov-based stability for the reduced-order systems.
\begin{theorem}
Consider the QB system~\eqref{eq:Quad_bilin_Sys} with a  stable matrix  $A$. Let $P_\cT$ and $Q_\cT$ be its truncated reachability and observability Gramians, defined in \Cref{coro:gram_trc_re}, respectively. If the reduced-order system is determined as shown in \Cref{algo:BT_qbdae}, then for a Lyapunov function $\cF(\hx) = \hx^T\Sigma_1\hx$, we have
$$\cF(\hx) >0,\qquad \tfrac{d}{dt}(\cF(\hx)) <0\qquad \forall~ \hx \in \cB_{0,r}\backslash \{0\},$$
where $r = \dfrac{\sigma_{\min}(\cV^T\cG\cV)}{2 \|\Sigma_1\| \|\hH\|}$
and $\cG =  \cH^{(2)}(P_1\otimes Q_1) \left(\cH^{(2)}\right)^T + \sum_{k=1}^mN_k^TQ_1N_k +C^TC$ with $P_1$ and $Q_1$ being  the solutions of~\eqref{eq:linear_CG} and \eqref{eq:linear_OG}, respectively.
\end{theorem}\label{thm:stability}
\begin{proof}
First, we establish the relation between $\cV$, $\cW$, $Q_\cT$ and $\Sigma_1$. For this, we consider
\begin{align*}
\cW\Sigma_1 &= R V_1 \Sigma_1^{\tfrac{1}{2}} = R V_1\bbm \Sigma_1 & 0 \ebm^T U^T U_1 \Sigma_1^{-\tfrac{1}{2}}= R V\Sigma U^T U_1 \Sigma_1^{-\tfrac{1}{2}}\\
& = R R^TS^T U_1 \Sigma_1^{-\tfrac{1}{2}} = Q_\cT\cV.
\end{align*}
Keeping in mind  the above relation, we get
\begin{equation}\label{eq:red_gram}
\begin{aligned}
&\hA^T \Sigma_1 + \Sigma_1 \hA + \cV ^T\cG\cV  = \cV^T A^T \cW\Sigma_1 + \Sigma_1 \cW^T A \cV + \cV^T\cG \cV \\
&\qquad = \cV^T A^T Q_\cT \cV +\cV^T Q_\cT A \cV + \cV^T \cG \cV  = \cV^T(A^TQ_\cT + Q_\cT A + \cG)\cV = 0.
\end{aligned}
\end{equation}
Since $\cG$ is a positive  semidefinite matrix and $\cV$ has full column rank,  $\cV^T\cG\cV$ is also positive semidefinite. This implies that $\eta(\hA)~\leq~0$, where $\eta(\cdot)$ denotes the spectral abscissa of a matrix. Coming back to the Lyapunov function $\cF(\hx) = \hx^T\Sigma_1 \hx$, which is always greater than $0$  for all $\hat{x}\not=0$ due to $\Sigma_1$ being a positive definite matrix, we compute the derivative of the Lyapunov function as
 \begin{equation*}
 \begin{aligned}
  \dfrac{d}{dt}\cF(\hx) &= \dot{\hx}^T\Sigma_1\hx + \hx^T\Sigma_1\dot{\hx} \\
  &= \hx^T \hA^T \Sigma_1 \hx +  (\hx^T\otimes \hx^T)\hH^T \Sigma_1\hx + \hx^T \Sigma_1\hA\hx + \hx^T \Sigma_1 \hH (\hx\otimes \hx)\\
  &= \hx^T(\hA^T\Sigma_1 + \Sigma_1\hA) \hx +  (\hx^T\otimes \hx^T)\hH^T \Sigma_1\hx  + \hx^T \Sigma_1 \hH (\hx\otimes \hx).
 \end{aligned}
 \end{equation*}
 Substituting $\hA^T\Sigma_1 + \Sigma_1\hA = - \cV^T\cG\cV$ from \eqref{eq:red_gram} in the above equation yields
 \begin{align}\label{eq:2}
  \tfrac{d}{dt}\cF(\hx) &=  -\hx^T\cV^T\cG\cV \hx +  2\hx^T \Sigma_1 \hH (\hx\otimes \hx).
 \end{align}
As
 \begin{align*}
  & \hx^T \cV^T\cG\cV \hx \geq  \sigma_{\text{min}}(\cV^T\cG\cV) \|\hx\|^2,
 \end{align*}
implying
  \begin{align*}
   &  -\hx^T\cV^T\cG\cV x \leq -\sigma_{\text{min}}(\cV^T\cG\cV) \|\hx\|^2,
  \end{align*}
 inserting the above inequality in~\eqref{eq:2} leads to
  \begin{align*}
   \tfrac{d}{dt}\cF(\hx) &\leq  -\sigma_{\text{min}}(\cV^T\cG\cV) \|\hx\|^2 +  2\|\hx\|^3 \|\Sigma_1\| \|\hH\|.
  \end{align*}
For locally asymptotic stability of the reduced-order system, we require
   \begin{align*}
   \tfrac{d}{dt}\cF(\hx) &\leq  -\sigma_{\text{min}}(\cV^T\cG\cV) \|\hx\|^2 +  2\|\hx\|^3 \|\Sigma_1\| \|\hH\| <0,
  \end{align*}
which gives rise to the following bound on $\|\hx\|$:
  $$\|\hx\| < \tfrac{\sigma_{\text{min}}(\cV^T\cG\cV)}{2 \|\Sigma_1\|  \|\hH\|}.$$
  This concludes the proof.
\end{proof}

\section{Numerical Experiments} \label{sec:numerics}\label{sec:Numerical}\def\redQBbal {\ensuremath{{\mathsf{RedQBbal}}}}
\def\redBbal {\ensuremath{{\mathsf{RedBbal}}}}

\def\GramQB {\ensuremath{{\mathsf{TGramQB}}}}
\def\GramB {\ensuremath{{\mathsf{TGramB}}}}

\definecolor{mycolor1}{rgb}{1.00000,0.00000,1.00000}%
\definecolor{mycolor2}{rgb}{0.00000,1.00000,1.00000}%

In this section, we consider MOR of several QB control systems and evaluate the efficiency of the proposed balanced truncation technique (\Cref{algo:BT_qbdae}). For this, we need to solve a number of conventional  Lyapunov equations. In our numerical experiments, we determine low-rank factors of these Lyapunov equations by using the ADI method as proposed in~\cite{benner2014self}. We compare the proposed methodology with the existing moment-matching techniques for QB systems, namely one-sided moment-matching~\cite{morGu09} and  its recent extension to two-sided moment-matching~\cite{morBenB15}. These moment-matching methods aim at approximating the underlying generalized transfer functions of the system. Moreover, we need interpolation points in order to apply the moment-matching methods; thus, we choose $l$ linear $\cH_2$-optimal interpolation points, determined by using \emph{IRKA}~\cite{morGugAB08} on the corresponding linear part. This leads to a reduced QB system of order $\hn = 2l$.  All the
simulations are done on
\matlab~Version 8.0.0.783(R2012b)64-bit(glnxa64) on a board with 4 \intel ~\xeon ~E7-8837 CPUs with a 2.67-GHz clock speed, 8 Cores each and 1TB of total RAM, openSUSE Linux 12.04.

\subsection{Nonlinear RC ladder}
As a first example, we discuss a nonlinear RC ladder. It is a well-known example and is used as one of the benchmark problems in the community of nonlinear model reduction; see, e.g.,~\cite{bai2002krylov,morBreD10,morGu09,li2005compact,morPhi03}.  A detailed description of the dynamics can be found in the mentioned references; therefore, we omit it for the brevity of the paper.  However, we like to comment on the nonlinearity present in the RC ladder. The nonlinearity arises from the presence of the diode I-V characteristic $i_D := e^{40v_D}{-} v_D {-}1$, where $v_D$ denotes the voltage across the diode. As shown in~\cite{morGu09},  introducing some appropriate new variables allows us to write the system dynamics in the QB form of dimension $n = 2 k$, where $k$ is the number of capacitors in the ladder.

\begin{figure}[!tb]
	\centering
	\setlength\fheight{3cm}
	\setlength\fwidth{6cm}
	\tikzsetnextfilename{Figures/RC_Sigma}%
%
%
\begin{tikzpicture}

\begin{axis}[%
width=\fwidth,
height=\fheight,
scale only axis,
xmin=0,
xmax=60,
ymode=log,
ymin=1e-20,
ymax=1,
ytick = {1.4e-4,1e-15},
yticklabels={1.4e-04,1e-15},
xtick = {10,45},
xticklabels={10,45},
ylabel = {$\sigma$}
]
\addplot [color=blue,solid,forget plot,line width = 1.2pt]
  table[row sep=crcr]{1	1\\
2	0.229562345662743\\
3	0.0663156577747169\\
4	0.0364527644588128\\
5	0.0111823801986215\\
6	0.00393305475145492\\
7	0.00209159251963463\\
8	0.00117617884015827\\
9	0.000412944558781164\\
10	0.000140631296881609\\
11	5.34951916058516e-05\\
12	3.26622115078762e-05\\
13	1.41611737415092e-05\\
14	5.05402258718164e-06\\
15	2.49562597405828e-06\\
16	1.42355529263872e-06\\
17	5.25900183135866e-07\\
18	2.17608018496604e-07\\
19	1.39922020586754e-07\\
20	5.54454948936037e-08\\
21	2.32879809923543e-08\\
22	1.4825835365824e-08\\
23	6.04687999008707e-09\\
24	2.68714964088052e-09\\
25	1.66818621976868e-09\\
26	6.70312975080217e-10\\
27	3.48465424289662e-10\\
28	1.92197106765617e-10\\
29	7.95281927995842e-11\\
30	4.75692576634679e-11\\
31	2.24715394594569e-11\\
32	1.06194050710299e-11\\
33	6.38665309069647e-12\\
34	2.72339508656769e-12\\
35	1.62066110700597e-12\\
36	7.85975717687222e-13\\
37	3.91051335860335e-13\\
38	2.28824938576364e-13\\
39	1.00140309574445e-13\\
40	5.98960518517007e-14\\
41	2.92386768812554e-14\\
42	1.25396870480611e-14\\
43	7.0914020405516e-15\\
44	3.65791088054997e-15\\
45	1.37917946848204e-15\\
46	1.57873791469124e-16\\
47	1.1333918884886e-16\\
48	9.03518288235931e-17\\
49	9.03518288235931e-17\\
50	9.03518288235931e-17\\
51	9.03518288235931e-17\\
52	9.03518288235931e-17\\
53	9.03518288235931e-17\\
54	9.03518288235931e-17\\
55	9.03518288235931e-17\\
56	9.03518288235931e-17\\
57	9.03518288235931e-17\\
58	9.03518288235931e-17\\
59	9.03518288235931e-17\\
60	9.03518288235931e-17\\
61	9.03518288235931e-17\\
62	9.03518288235931e-17\\
63	9.03518288235931e-17\\
64	9.03518288235931e-17\\
65	9.03518288235931e-17\\
66	9.03518288235931e-17\\
67	9.03518288235931e-17\\
68	9.03518288235931e-17\\
69	9.03518288235931e-17\\
70	9.03518288235931e-17\\
71	9.03518288235931e-17\\
72	9.03518288235931e-17\\
73	5.17087281923197e-17\\
74	4.28765382578134e-17\\
};
\addplot [color=green,dashed,forget plot,line width = 1.2pt]
  table[row sep=crcr]{0	1.4e-4\\
80	1.4e-4\\
};
\addplot [color=green,dashed,forget plot,line width = 1.2pt]
  table[row sep=crcr]{0	1e-15\\
80	1e-15\\
};
\addplot [color=green,dashed,forget plot,line width = 1.2pt]
  table[row sep=crcr]{10	1e0\\
10	1e-20\\
};
\addplot [color=green,dashed,forget plot,line width = 1.2pt]
  table[row sep=crcr]{45	1e0\\
45	1e-20\\
};
\end{axis}
\end{tikzpicture}%
%

	\caption{A RC ladder: decay of the normalized singular values based the truncated Gramians, and the dotted lines show the normalized singular value for $\hn =  10$ and the order of the reduced system corresponding to the normalized singular value  $1e{-15}$.}
	\label{fig:RC_sigma}
\end{figure}

We consider $500$ capacitors in the ladder, leading to a QB system of order $n=1000$. For this particular example, the matrix $A$ is a semi-stable matrix, i.e., $0\subset \sigma(A) $. As a result, the truncated Gramians of the system might not exist; therefore, we replace the matrix $A$ by $A_s := A {-} 0.05I$, where $I$ is the identity matrix, to determine these Gramians. However, note that we project the original system with the matrix $A$ to compute a reduced-order system  but the projection matrices are computed using the Gramians obtained via the shifted matrix $A_s$.  In \Cref{fig:RC_sigma}, we show the decay of the  singular values, determined by the truncated Gramians (with the shifted $A$). We then compute the reduced system of order $\hn = 10$ by using balanced truncation. Also, we determine $5$ $\cH_2$-optimal linear interpolation points and compute reduced-order systems of order $\hn =  10$ via one-sided and two-sided projection methods.

To compare the quality of these approximations, we simulate these systems for the input signals $u_1(t) = 5\left(\sin(2\pi/10) +1\right)$ and $u_2(t) = 10\left( t^2\exp(-t/5)\right)$.  \Cref{fig:RC_timedomain} presents the transient responses and relative errors of the output for these input signals, which shows that balanced truncation outperforms the one-sided interpolatory method; on the other hand, we see that balanced truncation is competitive   to the two-sided interpolatory projection for this example.

\begin{figure}[!tb]
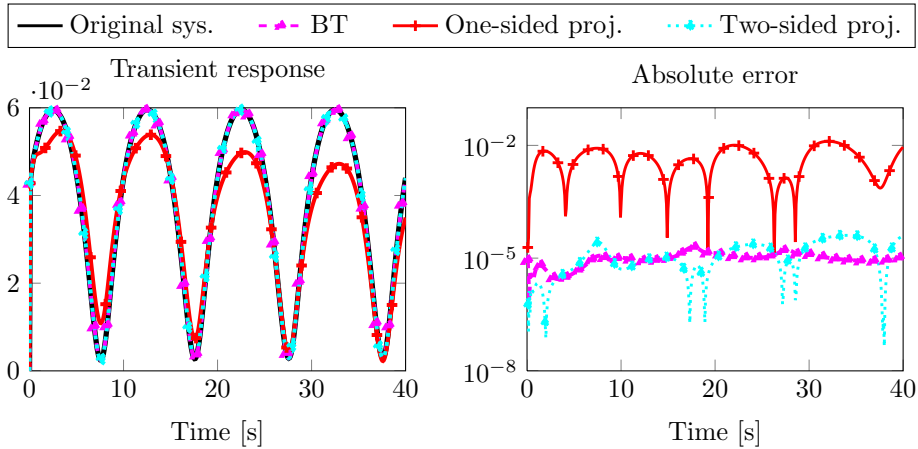
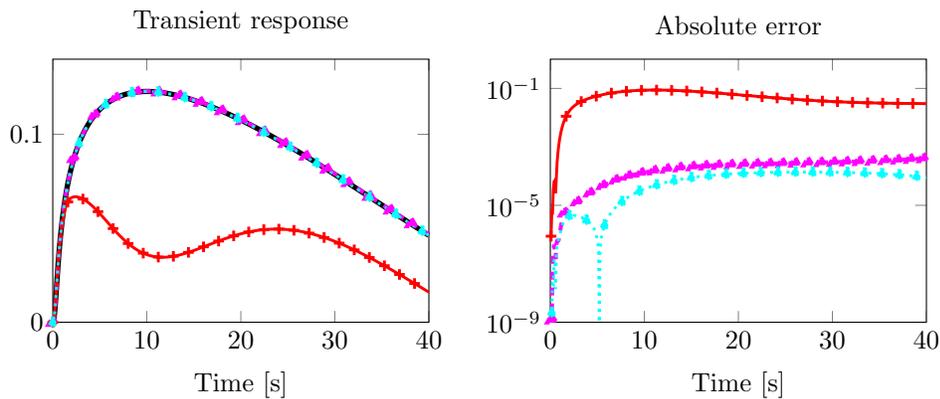

	\centering
	\begin{tikzpicture}
	\begin{customlegend}[legend columns=-1, legend style={/tikz/every even column/.append style={column sep=.5cm}} , legend entries={Original sys., BT, One-sided proj. , Two-sided proj.}, ]
	\addlegendimage{black,line width = 1.1pt}
	\addlegendimage{mycolor1,line width = 1.1pt,mark = triangle*, dashed}
	\addlegendimage{red ,line width = 1.1pt, mark = +}
	\addlegendimage{mycolor2,line width = 1.1pt, dotted, mark = diamond*}
	\end{customlegend}
	\end{tikzpicture}
	\begin{subfigure}[!htb]{\textwidth}
		\centering
		\setlength\fheight{3.5cm}
		\setlength\fwidth{5cm}
	\tikzsetnextfilename{FinalPics/RC_Input1_response}%
	\input{FinalPics/RC_Input1_response.tikz}%

		\centering
		\setlength\fheight{3.5cm}
		\setlength\fwidth{5cm}
	\tikzsetnextfilename{FinalPics/RC_Input1_error}%
	\input{FinalPics/RC_Input1_error.tikz}%

		\caption{Comparison of the original and the reduced-order systems for $u_1(t) = 5\left(\sin(2\pi/10) +1\right)$. }
	\end{subfigure}
		\begin{subfigure}[!tb]{\textwidth}
			\centering
			\setlength\fheight{3.5cm}
			\setlength\fwidth{5cm}
	\tikzsetnextfilename{FinalPics/RC_Input2_response}%
	\input{FinalPics/RC_Input2_response.tikz}%

			\centering
			\setlength\fheight{3.5cm}
			\setlength\fwidth{5cm}
	\tikzsetnextfilename{FinalPics/RC_Input2_error}%
	\input{FinalPics/RC_Input2_error.tikz}%

		\caption{Comparison of the original and the reduced-order systems for $u_2(t) = 10\left( t^2\exp(-t/5)\right)$. }
		\end{subfigure}
	\caption{A RC ladder: comparison of reduced-order systems obtained by balanced truncation (BT) and moment-matching methods for  two arbitrary control inputs.}
	\label{fig:RC_timedomain}
\end{figure}

\subsection{One-dimensional Chafee-Infante equation}
As a second example, we consider the one-dimensional Chafee-Infante (Allen-Cahn) equation. This nonlinear system has been widely studied in the literature; see, e.g.,~\cite{chafee1974bifurcation,hansen2012second}, and its model reduction related problem was recently considered in~\cite{morBenB15}. The governing equation, subject to  initial conditions and  boundary control,   have a cubic nonlinearity:
\begin{equation}\label{sys:Chafee_infante}
 \begin{aligned}
  \dot{v} +  v^3 &= v_{xx} + v, & ~~ &(0,L)\times (0,T),&\qquad
  v(0,\cdot)  &= u(t), & ~~ & (0,T),\\
  v_x(L,\cdot) &= 0, & ~~ & (0,T),&
  v(x,0) &= 0, & ~~ & (0,L).
 \end{aligned}
\end{equation}
Here, we  make use of a finite difference scheme and consider $k$ grid points in the spatial domain, leading to a semi-discretized nonlinear ODE. As shown in~\cite{morBenB15}, the smooth nonlinear system can be transformed into a QB system by introducing appropriate new state variables. Therefore, the system~\eqref{sys:Chafee_infante} with the cubic nonlinearity can be rewritten  in the QB form by defining  new variables $w_i = v_i^2$ with derivate  $\dot{w}_i = 2v_i\dot{v}_i$. We observe the response at the right boundary at $x = L$. We use the number of grid points $ k = 500$, which results in a QB system of  dimension $n = 2\cdot500 = 1000$  and  set the length $L = 1$. In \Cref{fig:1D_chafee_singular}, we show the decay of the normalized singular values based on the truncated Gramians of the system.

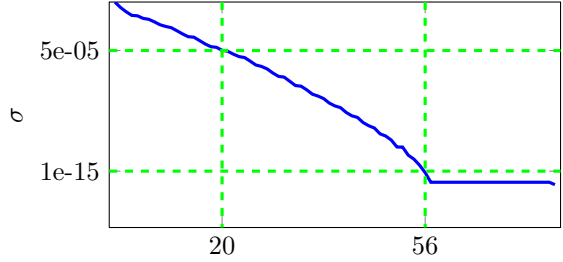
\begin{figure}[!tb]
\centering
	\setlength\fheight{3cm}
	\setlength\fwidth{6cm}
	\tikzsetnextfilename{Figures/Chafee_Sigma}%
%
%
\begin{tikzpicture}

\begin{axis}[%
width=\fwidth,
height=\fheight,
scale only axis,
xmin=0,
xmax=80,
ymode=log,
ymin=1e-20,
ymax=1,
ytick = {5e-5,1e-15},yticklabels={5e-05,1e-15},
xtick = {20,56},xticklabels={20,56},
ylabel = {$\sigma$}
]
\addplot [color=blue,solid,forget plot,line width = 1.2]
  table[row sep=crcr]{1	1\\
2	0.327721718984446\\
3	0.137147627740624\\
4	0.0653638328743227\\
5	0.0598908694083367\\
6	0.0357002372918929\\
7	0.0293372323338058\\
8	0.0174717542250776\\
9	0.00967101037961997\\
10	0.00684256551969583\\
11	0.00528570501244001\\
12	0.00282633801970931\\
13	0.00147933666302931\\
14	0.000843313060859994\\
15	0.000760347058219831\\
16	0.000384404950294705\\
17	0.000191925160651136\\
18	0.000109923130741453\\
19	9.59818944884668e-05\\
20	5.2503938687367e-05\\
21	4.32330427790912e-05\\
22	2.37292254991926e-05\\
23	1.45097090587808e-05\\
24	1.18319541107117e-05\\
25	5.63441943602702e-06\\
26	2.58524879458992e-06\\
27	1.9261217401414e-06\\
28	1.1543934594223e-06\\
29	5.05383645525665e-07\\
30	2.56438261095115e-07\\
31	2.18120210626264e-07\\
32	9.30799671073286e-08\\
33	3.92810760996664e-08\\
34	3.41994471396489e-08\\
35	1.63787003943334e-08\\
36	6.74387447848644e-09\\
37	4.56565460938058e-09\\
38	2.73881654394338e-09\\
39	1.0928712782247e-09\\
40	6.09757010008186e-10\\
41	4.25807398560836e-10\\
42	1.60914644479464e-10\\
43	8.12357134019808e-11\\
44	5.85653619701156e-11\\
45	2.03506071982201e-11\\
46	1.06369760850757e-11\\
47	6.68274826990649e-12\\
48	2.04421359593952e-12\\
49	1.30615798738923e-12\\
50	5.66844984193097e-13\\
51	1.38476238198049e-13\\
52	1.3569620428653e-13\\
53	2.60132188897984e-14\\
54	1.16911502653909e-14\\
55	3.41002783472923e-15\\
56	7.42417751286477e-16\\
57	1.03316375052987e-16\\
58	1.03316375052987e-16\\
59	1.03316375052987e-16\\
60	1.03316375052987e-16\\
61	1.03316375052987e-16\\
62	1.03316375052987e-16\\
63	1.03316375052987e-16\\
64	1.03316375052987e-16\\
65	1.03316375052987e-16\\
66	1.03316375052987e-16\\
67	1.03316375052987e-16\\
68	1.03316375052987e-16\\
69	1.03316375052987e-16\\
70	1.03316375052987e-16\\
71	1.03316375052987e-16\\
72	1.03316375052987e-16\\
73	1.03316375052987e-16\\
74	1.03316375052987e-16\\
75	1.03316375052987e-16\\
76	1.03316375052987e-16\\
77	1.03316375052987e-16\\
78	1.03316375052987e-16\\
79	6.34002089616109e-17\\
80	0\\
81	0\\
82	0\\
83	0\\
84	0\\
85	0\\
86	0\\
87	0\\
88	0\\
89	0\\
90	0\\
91	0\\
92	0\\
93	0\\
94	0\\
95	0\\
96	0\\
97	0\\
98	0\\
99	0\\
100	0\\
101	0\\
102	0\\
};
\addplot [color=green,dashed,forget plot,line width = 1.2]
  table[row sep=crcr]{0	5e-5\\
80	5e-5\\
};
\addplot [color=green,dashed,forget plot,line width = 1.2]
  table[row sep=crcr]{0	1e-15\\
80	1e-15\\
};
\addplot [color=green,dashed,forget plot,line width = 1.2]
  table[row sep=crcr]{20	1e0\\
20	1e-20\\
};
\addplot [color=green,dashed,forget plot,line width = 1.2]
  table[row sep=crcr]{56	1e0\\
56	1e-20\\
};
\end{axis}
\end{tikzpicture}%
%

		\caption{Chafee-Infante equation: decay of the normalized singular values based the truncated Gramians, and dotted line shows the normalized singular value  for $\hn =  20$ and the order of the reduced-order system corresponding to the normalized singular value  $1e{-15}$.}
		\label{fig:1D_chafee_singular}
\end{figure}

We determine reduced systems of order $\hn =   20$ by using balanced truncation, and one-sided and two-sided interpolatory projection methods. To compare the quality of these reduced-order systems, we observe the outputs of the original and reduced-order  systems for two arbitrary control inputs $u(t) = 5t\exp(-t)$ and $u(t) = 30(\sin(\pi t)+1)$ in \Cref{fig:time_domain_CI}.

\begin{figure}[!htb]
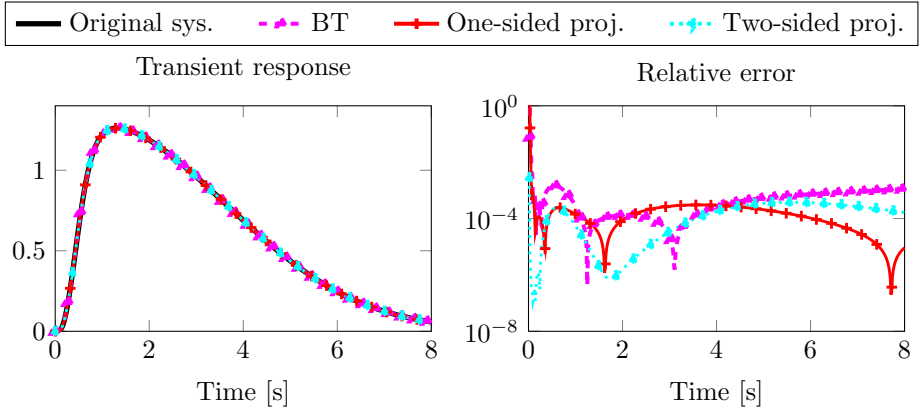
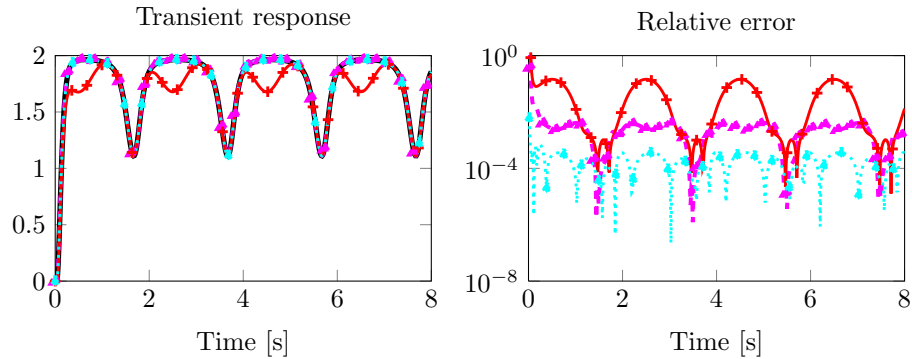

        \centering
        \begin{tikzpicture}
     \begin{customlegend}[legend columns=-1, legend style={/tikz/every even column/.append style={column sep=0.5cm}} , legend entries={Original sys., BT, One-sided proj. , Two-sided proj.}, ]
	\addlegendimage{black,line width = 2pt}
	\addlegendimage{mycolor1,line width = 1.5pt,mark = triangle*, dashed}
	\addlegendimage{red ,line width = 1.5pt, mark = +}
	\addlegendimage{mycolor2,line width = 1.5pt, dotted, mark = diamond*}
     \end{customlegend}
 \end{tikzpicture}
 \begin{subfigure}[!b]{\textwidth}
  \centering
		\setlength\fheight{3cm}
	\setlength\fwidth{5.0cm}
	\tikzsetnextfilename{FinalPics/Chafee_Input2_response}%
	\input{FinalPics/Chafee_Input2_response.tikz}%

	\tikzsetnextfilename{FinalPics/Chafee_Input2_error}%
	\input{FinalPics/Chafee_Input2_error.tikz}%

			\caption{Comparison of the original and the reduced-order systems for $u_1(t) = 5\left.t\exp(-t)\right.$. }
			\label{fig:chafee_input2}
 \end{subfigure}
  \begin{subfigure}[!tb]{\textwidth}
  	\centering
  	\setlength\fheight{3cm}
  	\setlength\fwidth{5.0cm}
	\tikzsetnextfilename{FinalPics/Chafee_Input1_response}%
	\input{FinalPics/Chafee_Input1_response.tikz}%

	\tikzsetnextfilename{FinalPics/Chafee_Input1_error}%
	\input{FinalPics/Chafee_Input1_error.tikz}%

  				\caption{Comparison of the original and the reduced-order systems for $u_2(t) = 30\left(\sin(\pi t)+1\right)$. }
  \end{subfigure}
 \caption{Chafee-Infante equation: comparison of the reduced-order systems obtained via balanced truncation and moment-matching methods for the inputs $u_1(t) = 5\left(t\exp(-t)\right)$  and $u_2(t) = 30\left(\sin(\pi t)+1\right)$. }
\label{fig:time_domain_CI}
\end{figure}

\Cref{fig:chafee_input2} shows that the reduced systems obtained via balanced truncation and one-sided and two-sided interpolatory projection methods are almost  of the same quality for input $u_1$. But for the input $u_2$, the reduced system obtained via the one-sided interpolatory projection method completely fails to capture the dynamics of the system, while balanced truncation and two-sided interpolatory projection can reproduce the system dynamics with a slight advantage of two-sided projection regarding accuracy.

However, it is worthwhile to mention  that as we increase the order of the reduced system, the two-sided interpolatory projection method tends to produce unstable reduced-order systems. On the other hand, the accuracies of the reduced-order systems obtained by balanced truncation and one-sided moment-matching increase with the order of the reduced systems.

\subsection{The FitzHugh-Nagumo (F-N) system}
Lastly, we consider the F-N system, a simplified neuron model of the Hodgkin-Huxley model, which describes activation and deactivation dynamics of a spiking neuron. This model has been considered in the framework of POD-based~\cite{morChaS10} and moment-matching model reduction techniques~\cite{morBenB12a}. The dynamics of the system is governed by the following nonlinear coupled differential equations:
\begin{equation}
\begin{aligned}
\epsilon v_t(x,t) & =\epsilon^2v_{xx}(x,t) + f(v(x,t)) -w(x,t) + q,\\
w_t(x,t) &= hv(x,t) -\gamma w(x,t) + q
\end{aligned}
\end{equation}
with a nonlinear function $f(v(x,t)) = v(v-0.1)(1-v)$  and the initial and boundary conditions:
\begin{equation}
\begin{aligned}
v(x,0) &=0, &\quad w(x,0)&=0, \quad & &x\in [0,L],\\
v_x(0,t) &= i_0(t), & v_x(1,t) &= 0, & &t\geq 0,
\end{aligned}
\end{equation}
where $\epsilon = 0.015,~h=0.5,~\gamma = 2,~q = 0.05$. We set the length $L = 0.2$. The stimulus $i_0$ acts as an actuator, taking the values $i_0(t) = 5\cdot 10^4t^3\exp(-15t)$, and the variables $v$ and $w$ denote the voltage and recovery voltage, respectively. We also assume the same outputs of interest as considered in~\cite{morBenB12a}, which are $v(0,t)$ and $w(0,t)$. These outputs describe nothing but the limit cyclic at the left boundary. Using a finite difference discretization scheme, one can obtain a system with two inputs and two outputs of  dimension $2 k$ with cubic nonlinearities, where $k$ is the number of degrees of freedom.   Similar to the previous  example, the F-H system can also be transformed into a QB system of dimension $n = 3 k$ by introducing a new state variable $z_i = v_i^2$. We set $k = 500$,  resulting in a QB system of order $n = 1500$.  \Cref{fig:Fitz_sigma} shows the decay of the singular values  based on the truncated Gramians for the QB system.
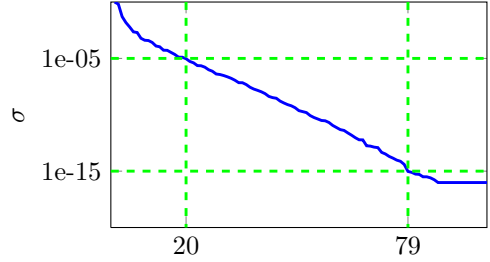
\begin{figure}[tb!]
        \centering
  \centering
  	\setlength\fheight{3cm}
	\setlength\fwidth{5cm}
	\tikzsetnextfilename{Figures/Fitz_Sigma}%
%
%
\begin{tikzpicture}

\begin{axis}[%
width=\fwidth,
height=\fheight,
scale only axis,
xmin=0,
xmax=100,
ymode=log,
ymin=1e-20,
ymax=1,
ytick = {1e-5,1e-15},
yticklabels={1e-05,1e-15},
xtick = {20,79},
xticklabels={20,79},
ylabel = {$\sigma$}
]
\addplot [color=blue,solid,forget plot,line width = 1.1pt]
  table[row sep=crcr]{1	1\\
2	0.647281837956837\\
3	0.0510042741209191\\
4	0.0151544761715599\\
5	0.00610836068256647\\
6	0.00225571457115442\\
7	0.00197676925611745\\
8	0.000613623968095761\\
9	0.000408558284908071\\
10	0.000383699917064134\\
11	0.000260347085214034\\
12	0.000125831635871974\\
13	0.000106890524328919\\
14	5.33331901147038e-05\\
15	5.19364440293642e-05\\
16	3.39600188780222e-05\\
17	2.24767460525141e-05\\
18	1.28348203262763e-05\\
19	1.27453809560998e-05\\
20	9.32210107521646e-06\\
21	5.37735908049636e-06\\
22	4.025376826218e-06\\
23	2.30809860623783e-06\\
24	2.21662966174914e-06\\
25	1.6641618279093e-06\\
26	9.85548583442404e-07\\
27	7.27202615039302e-07\\
28	4.14306294100142e-07\\
29	3.730403486153e-07\\
30	3.02408233763846e-07\\
31	2.32337500386856e-07\\
32	1.7098438890812e-07\\
33	1.15267296432287e-07\\
34	7.03849228168793e-08\\
35	6.94814901388299e-08\\
36	4.78913081557202e-08\\
37	2.78091729055474e-08\\
38	1.91405288981787e-08\\
39	1.47060372156334e-08\\
40	1.09736773444372e-08\\
41	7.53710705349778e-09\\
42	4.27106453771568e-09\\
43	3.33136060106473e-09\\
44	2.91862766933469e-09\\
45	1.63887027302277e-09\\
46	1.11521778980956e-09\\
47	7.95397328077599e-10\\
48	6.22371430728967e-10\\
49	4.20283228573783e-10\\
50	2.32203049427715e-10\\
51	1.90591885789577e-10\\
52	1.56361586707744e-10\\
53	8.494819663811e-11\\
54	5.73523903519026e-11\\
55	4.25465992125319e-11\\
56	3.12406899637513e-11\\
57	2.92250383726377e-11\\
58	2.06375926274038e-11\\
59	1.13083318457968e-11\\
60	7.38524858215203e-12\\
61	4.8217545771334e-12\\
62	4.08510469065876e-12\\
63	2.46757230559906e-12\\
64	1.549583259338e-12\\
65	9.27067837769274e-13\\
66	6.05500090069775e-13\\
67	5.91591341116872e-13\\
68	1.78238499887275e-13\\
69	1.66080611807852e-13\\
70	1.34898903582728e-13\\
71	1.21878322965173e-13\\
72	4.62957103817883e-14\\
73	3.33924787186436e-14\\
74	2.01766434613452e-14\\
75	1.43097381529916e-14\\
76	9.55871334946234e-15\\
77	7.4868167397243e-15\\
78	3.96507319093518e-15\\
79	9.39123630628194e-16\\
80	7.77178706794305e-16\\
81	5.495778304805e-16\\
82	5.15586091482086e-16\\
83	3.15737143683964e-16\\
84	3.14897583902224e-16\\
85	2.53732301255341e-16\\
86	1.71840523916832e-16\\
87	9.75296641928983e-17\\
88	9.75296641928983e-17\\
89	9.75296641928983e-17\\
90	9.75296641928983e-17\\
91	9.75296641928983e-17\\
92	9.75296641928983e-17\\
93	9.75296641928983e-17\\
94	9.75296641928983e-17\\
95	9.75296641928983e-17\\
96	9.75296641928983e-17\\
97	9.75296641928983e-17\\
98	9.75296641928983e-17\\
99	9.75296641928983e-17\\
100	9.75296641928983e-17\\
101	9.75296641928983e-17\\
102	9.75296641928983e-17\\
103	9.75296641928983e-17\\
104	9.75296641928983e-17\\
105	9.75296641928983e-17\\
106	9.75296641928983e-17\\
107	9.75296641928983e-17\\
108	9.75296641928983e-17\\
109	9.75296641928983e-17\\
110	9.75296641928983e-17\\
111	9.75296641928983e-17\\
112	9.75296641928983e-17\\
113	9.75296641928983e-17\\
114	9.75296641928983e-17\\
115	9.75296641928983e-17\\
116	9.75296641928983e-17\\
117	9.75296641928983e-17\\
118	9.75296641928983e-17\\
119	9.75296641928983e-17\\
120	9.75296641928983e-17\\
121	9.75296641928983e-17\\
122	9.75296641928983e-17\\
123	9.75296641928983e-17\\
124	9.75296641928983e-17\\
125	9.75296641928983e-17\\
126	9.75296641928983e-17\\
127	9.1534817354477e-17\\
128	8.86951945052927e-17\\
129	4.45613230989647e-17\\
};
\addplot [color=green,dashed,forget plot,line width = 1.1pt]
  table[row sep=crcr]{0	1e-5\\
100	1e-5\\
};
\addplot [color=green,dashed,forget plot,line width = 1.1pt]
  table[row sep=crcr]{0	1e-15\\
100	1e-15\\
};
\addplot [color=green,dashed,forget plot,line width = 1.1pt]
  table[row sep=crcr]{20	1e0\\
20	1e-20\\
};
\addplot [color=green,dashed,forget plot,line width = 1.1pt]
  table[row sep=crcr]{79	1e0\\
79	1e-20\\
};
\end{axis}
\end{tikzpicture}%
%

\caption{Decay of the normalized singular values based the truncated Gramians of the system for the F-N example, and the dotted lines show the normalized singular value for $\hn =  20$ and the order of the reduced system corresponding to the normalized singular value  $1e{-15}$.}
	\label{fig:Fitz_sigma}

\end{figure}

Furthermore, we determine reduced-order systems of order $\hn =  20 $ by using balanced truncation and moment-matching methods. We observe that the reduced-order systems, obtained via the moment-matching methods with linear $\cH_2$-optimal interpolations, both one-sided and two-sided, fail to capture the dynamics and limit cycles. We made several attempts to adjust the order of the reduced systems; but, we were unable to determine a stable reduced-order system via these methods with linear $\cH_2$-optimal points which could replicate the dynamics. Contrary to these methods,  the balanced truncation replicates the dynamics of the system faithfully as  can be seen in \Cref{fig:Fitz_output}.   Note that the reduced-order model reported in \cite{morBenB12a} was obtained using higher-order moments in a trial-and-error fashion but cannot be reproduced by an automated algorithm. As the dynamics of the system produces limit cycles for each spatial variable $x$, we, therefore, plot the solutions $v$ and $w$ over the
spatial domain $x$, which is also captured by the reduced-order system very well.
\begin{figure}[!tb]
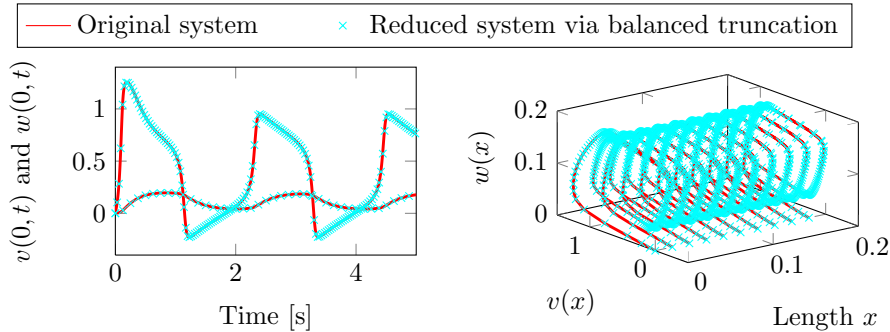

        \centering
        \begin{center}
         \begin{tikzpicture}
    \begin{customlegend}[legend columns=0, legend style={/tikz/every even column/.append style={column sep=1cm}} , legend entries={Original system,~ Reduced system via balanced truncation }, ]
    \addlegendimage{red}
    \addlegendimage{mycolor2,only marks, mark = x}
    \end{customlegend}
\end{tikzpicture}
\end{center}
\begin{subfigure}[t]{0.49\textwidth}
	\centering
	\setlength\fheight{2.5cm}
	\setlength\fwidth{4cm}
	\tikzsetnextfilename{Figures/Fitz_time_domain}%
	\input{Figures/Fitz_time_domain.tikz}%

	\caption{The response $v(t)$ and $w(t)$ at the left boundary.}
	\label{fig:Fitz_output}
\end{subfigure}
 \begin{subfigure}[t]{0.45\textwidth}
  \centering
  	\setlength\fheight{2.5cm}
	\setlength\fwidth{4.cm}
	\tikzsetnextfilename{Figures/Fitz_full_state}%
	\input{Figures/Fitz_full_state.tikz}%

{\caption{ Limit-cycles.}}
 \end{subfigure}
 \caption{FitzHugh-Nagumo system: comparison of the response at the left boundary and the limit cycle behavior of the original system and the reduced-order (balanced truncation) system. The reduced-order systems determined by moment-matching methods were unable to produce these limit cycles.}
 \label{fig:fitz_limit}
\end{figure}

\section{Conclusions}\label{sec:conclusions} In this paper, we have investigated balanced truncation model reduction for  QB control systems. We have proposed reachability and observability Gramians for QB systems based on the kernels of their underlying Volterra series.  Additionally, we have also introduced a truncated version of the Gramians. We, furthermore, have compared the controllability and observability energy functionals of the QB system with the quadratic forms of the proposed Gramians for the system and also investigated the connection between the Gramians and reachability/observability of the QB system. Also, we have discussed the advantages of the truncated version of Gramians in the MOR framework and  studied local Lyapunov stability of the reduced-order systems, obtained via the square-root balanced truncation. By means of various semi-discretized nonlinear PDEs, we have demonstrated the efficiency of the proposed balanced truncation methods for QB systems and compared it with the existing moment-matching techniques.
\bibliographystyle{siam}
\bibliography{mor}
\appendix
\section{A Convergence Result} 
\begin{lemma}\label{appendix_convergence}
	Consider a recurrence formula as follows:
	\begin{equation}\label{eq:sequence}
	x_{k+1} = F(x_k), \quad\forall \quad k\geq 1,
	\end{equation}
	where $F(x) = ax^2+bx+c$ and $a$, $b$, $c$ are real positive scaler numbers. Moreover, assume that $x_1 = c$.   Then, $\lim\nolimits_{k\rightarrow \infty}x_k =:x^*$ is finite if 
	\begin{subequations}
		\begin{align}
		b&<1, \quad \text{and} \label{cond1}\\
1>(b-1)^2-4ac &> 0.\label{cond2} 
\end{align}
	\end{subequations}
Furthermore,  $x^*$ is given by the smaller root of the the following quadratic equation:
$$ax^2 + (b-1)x +c = 0, \quad\text{i.e.,}$$
	\begin{equation}
	x^* = \dfrac{ -(b-1) - \sqrt{(b-1)^2-4ac}}{2a}.
	\end{equation}
	\end{lemma} 
	\begin{proof}
		First, note that the sequence \eqref{eq:sequence} contains only  real positive numbers. Thus, the equilibrium point  must also be  a real positive number. Furthermore, the equilibrium points solve the quadratic equation $F(x)-x = 0$, and we denote these equilibrium points by $x^{(1)}$ and $x^{(2)}$ with $x^{(1)}\leq x^{(2)}$.  Since $a$, $b$ and $c$ all are positive, both equilibrium points either can be positive or negative depending on the value of $b$.  To ensure the equilibrium points being positive, the minima of $F(x)-x$ must lie in the right half plane; thus, $b-1<0$, leading to the condition \eqref{cond1}. 
		
		Furthermore, we consider the derivative of $F(x)$, that is, $F'(x) := 2ax+b$.   Since $F'(x)$ is an increasing function and $F'(x) \geq 0$ $ \forall x \in [c,x^{(1)}]$, we have for $y \in [c,x^{(1)}]$:
		\begin{align*}
		 F'(y) &\leq  F'(x^{(1)})\\
		 &\leq 2ax^{(1)}+b = 2a\left(\dfrac{-(b-1) - \sqrt{(b-1)^2-4ac}}{2a} \right)  + b 
		  \leq 1 - \sqrt{(b-1)^2-4ac} .
		\end{align*}
	  Assuming  $1>(b-1)^2-4ac >0$, we have  $F'(y)  <1$, $\forall y\in [c,x^{(1)}]$.  Thus, by Banach fix-point theorem,  $F(x)$ is a contraction on $[c,x^{(1)}]$, and the fixed point is given by $x^{(1)}$.
\end{proof}

\end{document}